\documentclass[a4paper, leqno, 11pt]{amsart}

\usepackage[T1]{fontenc}	
\usepackage{graphicx}
\usepackage{amssymb,centernot}
\usepackage{upgreek}
\usepackage{mathrsfs}
\usepackage[usenames,dvipsnames]{xcolor}
\usepackage[inline,shortlabels]{enumitem}
\usepackage[sort,numbers]{natbib}						
\usepackage{verbatim}								
\usepackage{dsfont}									

\usepackage{pdflscape}
\usepackage{afterpage}

\usepackage{anyfontsize}
\usepackage[
						colorlinks,				
						breaklinks,unicode,
						hypertexnames=false,
						citecolor=OliveGreen,
						linkcolor=Maroon
						]{hyperref} 

\usepackage{multirow}

\usepackage{xypic}
\usepackage{mathtools}
\usepackage{xspace}

\usepackage[a4paper,top=3.2cm,bottom=2.7cm,left=3cm,right=3cm, bindingoffset=0mm]{geometry}
\usepackage{booktabs}

\usepackage{tikz}

\usetikzlibrary{calc, positioning, shapes.geometric, patterns, cd}

\newlength{\hatchspread}
\newlength{\hatchthickness}
\newlength{\hatchshift}
\newcommand{\hatchcolor}{}
\tikzset{hatchspread/.code={\setlength{\hatchspread}{#1}},
         hatchthickness/.code={\setlength{\hatchthickness}{#1}},
         hatchshift/.code={\setlength{\hatchshift}{#1}},
         hatchcolor/.code={\renewcommand{\hatchcolor}{#1}}}
\tikzset{hatchspread=3pt,
         hatchthickness=0.4pt,
         hatchshift=0pt,
         hatchcolor=black}
\pgfdeclarepatternformonly[\hatchspread,\hatchthickness,\hatchshift,\hatchcolor]
   {custom north west lines}
   {\pgfqpoint{\dimexpr-2\hatchthickness}{\dimexpr-2\hatchthickness}}
   {\pgfqpoint{\dimexpr\hatchspread+2\hatchthickness}{\dimexpr\hatchspread+2\hatchthickness}}
   {\pgfqpoint{\dimexpr\hatchspread}{\dimexpr\hatchspread}}
   {
    \pgfsetlinewidth{\hatchthickness}
    \pgfpathmoveto{\pgfqpoint{0pt}{\dimexpr\hatchspread+\hatchshift}}
    \pgfpathlineto{\pgfqpoint{\dimexpr\hatchspread+0.15pt+\hatchshift}{-0.15pt}}
    \ifdim \hatchshift > 0pt
      \pgfpathmoveto{\pgfqpoint{0pt}{\hatchshift}}
      \pgfpathlineto{\pgfqpoint{\dimexpr0.15pt+\hatchshift}{-0.15pt}}
    \fi
    \pgfsetstrokecolor{\hatchcolor}
    \pgfusepath{stroke}
   }

\pgfdeclarepatternformonly[\hatchspread,\hatchthickness,\hatchshift,\hatchcolor]
   {custom north east lines}
   {\pgfqpoint{\dimexpr-2\hatchthickness}{\dimexpr-2\hatchthickness}}
   {\pgfqpoint{\dimexpr\hatchspread+2\hatchthickness}{\dimexpr\hatchspread+2\hatchthickness}}
   {\pgfqpoint{\dimexpr\hatchspread}{\dimexpr\hatchspread}}
   {
    \pgfsetlinewidth{\hatchthickness}
    \pgfpathmoveto{\pgfqpoint{\dimexpr\hatchshift-0.15pt}{-0.15pt}}
    \pgfpathlineto{\pgfqpoint{\dimexpr\hatchspread+0.15pt}{\dimexpr\hatchspread-\hatchshift+0.15pt}}
    \ifdim \hatchshift > 0pt
      \pgfpathmoveto{\pgfqpoint{-0.15pt}{\dimexpr\hatchspread-\hatchshift-0.15pt}}
      \pgfpathlineto{\pgfqpoint{\dimexpr\hatchshift+0.15pt}{\dimexpr\hatchspread+0.15pt}}
    \fi
    \pgfsetstrokecolor{\hatchcolor}
    \pgfusepath{stroke}
   }

\makeatletter
\newcommand*{\centerfloat}{%
  \parindent \z@
  \leftskip \z@ \@plus 1fil \@minus \textwidth
  \rightskip\leftskip
  \parfillskip \z@skip}
\makeatother

\newcommand{\Alpha}{\mathrm{A}}
\newcommand{\Beta}{\mathrm{B}}

\usepackage{xparse}

\ExplSyntaxOn
\NewDocumentCommand{\makeabbrev}{mmm}
 {
  \yoruk_makeabbrev:nnn { #1 } { #2 } { #3 }
 }

\cs_new_protected:Npn \yoruk_makeabbrev:nnn #1 #2 #3
 {
  \clist_map_inline:nn { #3 }
   {
    \cs_new_protected:cpn { #2 } { #1 { ##1 } }
   }
 }
 \ExplSyntaxOff

\makeabbrev{\textbf}{tbf#1}{a,b,c,d,e,f,g,h,i,j,k,l,m,n,o,p,q,r,s,t,u,v,w,x,y,z,A,B,C,D,E,F,G,H,I,J,K,L,M,N,O,P,Q,R,S,T,U,V,W,X,Y,Z}

\makeabbrev{\textbf}{bf#1}{a,b,c,d,e,f,g,h,i,j,k,l,m,n,o,p,q,r,s,t,u,v,w,x,y,z,A,B,C,D,E,F,G,H,I,J,K,L,M,N,O,P,Q,R,S,T,U,V,W,X,Y,Z}

\makeabbrev{\textsf}{tsf#1}{a,b,c,d,e,f,g,h,i,j,k,l,m,n,o,p,q,r,s,t,u,v,w,x,y,z,A,B,C,D,E,F,G,H,I,J,K,L,M,N,O,P,Q,R,S,T,U,V,W,X,Y,Z}

\makeabbrev{\mathsf}{mss#1}{a,b,c,d,e,f,g,h,i,j,k,l,m,n,o,p,q,r,s,t,u,v,w,x,y,z,A,B,C,D,E,F,G,H,I,J,K,L,M,N,O,P,Q,R,S,T,U,V,W,X,Y,Z}

\makeabbrev{\mathfrak}{mf#1}{a,b,c,d,e,f,g,h,i,j,k,l,m,n,o,p,q,r,s,t,u,v,w,x,y,z,A,B,C,D,E,F,G,H,I,J,K,L,M,N,O,P,Q,R,S,T,U,V,W,X,Y,Z}

\makeabbrev{\mathrm}{mrm#1}{a,b,c,d,e,f,g,h,i,j,k,l,m,n,o,p,q,r,s,t,u,v,w,x,y,z,A,B,C,D,E,F,G,H,I,J,K,L,M,N,O,P,Q,R,S,T,U,V,W,X,Y,Z}

\makeabbrev{\mathbf}{mbf#1}{a,b,c,d,e,f,g,h,i,j,k,l,m,n,o,p,q,r,s,t,u,v,w,x,y,z,A,B,C,D,E,F,G,H,I,J,K,L,M,N,O,P,Q,R,S,T,U,V,W,X,Y,Z}

\makeabbrev{\mathcal}{mc#1}{A,B,C,D,E,F,G,H,I,J,K,L,M,N,O,P,Q,R,S,T,U,V,W,X,Y,Z}

\makeabbrev{\mathbb}{mbb#1}{A,B,C,D,E,F,G,H,I,J,K,L,M,N,O,P,Q,R,S,T,U,V,W,X,Y,Z}

\makeabbrev{\mathscr}{ms#1}{A,B,C,D,E,F,G,H,I,J,K,L,M,N,O,P,Q,R,S,T,U,V,W,X,Y,Z}

\makeabbrev{\mathrm}{#1}{
Id,id,ran,rk,diag,stab,ann,conv,pr,ev,tr,End,Hom,sgn,im,op,can,fin,ext,red,tot,
rot,usc,lsc,Lip,LocLip,lip,bSymLip,osc,AC,loc,uloc,spec,coz,z,ul,
supp,Opt,Adm,Cpl,Geo,GeoSel,GeoOpt,GeoAdm,GeoCpl,reg,
bd,co,Ric,Exp,dExp,dist,seg,Seg,cut,fcut,Cut,SDiff,Iso,Isom,diam,cl,Homeo,Diff,Der,vol,dvol,inj,relint, Graph, sub,codim,
var,law,Var,Poi,Gam,pa,so,iso,fs,inv,pqi,mix,
TestF,
}

\makeabbrev{\mathsf}{#1}{DP,CD,BE,MCP,Ent,wMTW,MTW,RCD,ncRCD,QCD,EVI,Irr,IH,SC,wFe,VA,UP,Curv,Alex,CAT}

\newcommand{\bLip}{\mathrm{Lip}_b}

\newcommand{\T}{\tau} 
\newcommand{\A}{\Sigma} 
\newcommand{\Bo}[1]{\msB_{#1}} 

\newcommand{\Bdd}[1]{\msO_{#1}}
\newcommand{\Ed}{{\msE_\mssd}}  

\renewcommand{\complement}{\mathrm{c}}

\newcommand{\emparg}{{\,\cdot\,}}

\newcommand{\as}[1]{\quad #1\text{-a.e.}}

\newcommand{\dom}[1]{\msD(#1)}
\newcommand{\sem}[1]{\{#1\}_{t > 0}}
\newcommand{\reso}[1]{\{#1\}_{\alpha > 0}}

\newcommand{\Lipu}{\mathrm{Lip}^1}

\DeclareMathOperator{\eqdef}{\coloneqq}

\let\epsilon\varepsilon

\let\temp\phi
\let\phi\varphi
\let\varphi\temp

\newcommand{\rar}{\rightarrow}

\newcommand{\diff}{\mathop{}\!\mathrm{d}}

\newcommand{\tabs}[1]{\big\lvert#1\big\rvert}	

\newcommand{\norm}[1]{\left\lVert#1\right\rVert}					
\newcommand{\set}[1]{\left\{#1\right\}}							
\newcommand{\paren}[1]{\left(#1\right)}							
\newcommand{\tparen}[1]{\big({#1}\big)}

\newcommand{\tonde}[1]{\left(#1\right)}	
\newcommand{\ttonde}[1]{\big({#1}\big)}
\newcommand{\quadre}[1]{\left[#1\right]}							
\newcommand{\class}[2][]{\left[#2\right]_{#1}}						
\newcommand{\tclass}[2][]{\big [#2\big]_{#1}}						
\newcommand{\ttclass}[2][]{[#2]_{#1}}							

\newcommand{\tym}[1]{{\scriptscriptstyle{\times #1}}}
\newcommand{\otym}[1]{{\scriptscriptstyle{\otimes #1}}}
\newcommand{\osym}[1]{{\scriptscriptstyle{\odot #1}}}

\DeclareSymbolFont{symbolsC}{U}{pxsyc}{m}{n}
\SetSymbolFont{symbolsC}{bold}{U}{pxsyc}{bx}{n}
\DeclareFontSubstitution{U}{pxsyc}{m}{n}
\DeclareMathSymbol{\medcirc}{\mathbin}{symbolsC}{7}
\DeclareSymbolFont{symbolsZ}{OMS}{pxsy}{m}{n}
\SetSymbolFont{symbolsZ}{bold}{OMS}{pxsy}{bx}{n}
\DeclareFontSubstitution{OMS}{pxsy}{m}{n}

\newcommand{\seq}[1]{\paren{#1}}								
\newcommand{\tseq}[1]{{\big(#1\big)}}

\newcommand{\Cb}{\mcC_b}									

\newcommand{\Meas}{\mathscr M}

\newcommand{\pfwd}{\sharp}
\DeclareMathOperator*{\esssup}{esssup}
\DeclareMathOperator*{\essinf}{essinf}

\DeclareMathOperator{\car}{\mathds 1}

\DeclareMathOperator{\emp}{\varnothing} 
\newcommand{\N}{{\mathbb N}}
\newcommand{\R}{{\mathbb R}}

\newcommand{\TLDS}{\textsc{tlds}\xspace}

\newcommand{\parEMLDS}{\textsc{mlds}\xspace}

\newcommand{\quot}{{\sf P}}

\newcommand{\mrestr}[1]{\!\downharpoonright_{#1}}
\newcommand{\trid}{{\star}}

\allowdisplaybreaks

\usetikzlibrary{shapes.misc}
\tikzset{cross/.style={cross out, draw=black, minimum size=2*(#1-\pgflinewidth), inner sep=0pt, outer sep=0pt},
cross/.default={4pt}}

\newcommand{\iref}[1]{\ref{#1}}

\newcommand{\comma}{\,\,\mathrm{,}\;\,}
\newcommand{\semicolon}{\,\,\mathrm{;}\;\,}
\newcommand{\fstop}{\,\,\mathrm{.}}

\newcommand{\cdc}{\Gamma}
\newcommand{\LL}[2]{\mcL^{#1, #2}}
\newcommand{\TT}[2]{\mcT^{#1, #2}}

\newcommand{\EE}[2]{\mcE^{#1, #2}}

\newcommand{\SF}[2]{\cdc^{#1, #2}}

\usepackage{scrextend}						

\newcommand{\Cyl}[1]{\mcF^\dUpsilon\mcC^\infty_b(#1)}
\newcommand{\CylQP}[2]{\mcF^\dUpsilon\mcC^{\infty}_b(#2)_{#1}}

\newcommand{\Dz}{\mcD} 
 
\newcommand{\PP}{{\pi}}

\newcommand{\cpl}{q}
\newcommand{\QP}{{\mu}}

\newcommand{\coK}[2]{\co_{#2}({#1})}

\newcommand{\EN}{\overline{\N}}
\newcommand{\dUpsilon}{{\boldsymbol\Upsilon}}
\newcommand{\proj}{{\sf proj}}

\newcommand{\U}{\dUpsilon}
\newcommand{\E}{\mathcal E}
\newcommand{\F}{\mathcal F}
\renewcommand{\1}{\mathbf 1}

\newcommand{\CYL}{\mcF^\dUpsilon\mcC^\infty_b}
\newcommand{\comm}{\,\,\mathrm{,}\;\,}

\numberwithin{equation}{section}
\theoremstyle{plain}	
\newtheorem{thm}{Theorem}[section]
\newtheorem*{thm*}{Theorem}
\newtheorem*{mthm*}{Main Theorem}
\newtheorem{theorem}{Theorem}
\newtheorem{corollary}{Corollary}

\newtheorem{prop}[thm]{Proposition}
\newtheorem{lem}[thm]{Lemma}
\newtheorem{cor}[thm]{Corollary}
\newtheorem*{cor*}{Corollary}

\theoremstyle{definition}
\newtheorem{defs}[thm]{Definition}
\newtheorem*{defs*}{Definition}

\theoremstyle{remark}

\newtheorem{rem}[thm]{\bf Remark}
\newtheorem{ese}[thm]{\bf Example}
\newtheorem{ass}[thm]{\bf Assumption}

\newtheorem*{asm*}{{\bf List of Assumptions}}

\renewcommand{\paragraph}[1]{\medskip\emph{#1}.\quad}
\renewcommand{\#}{\sharp}

\makeatletter
\@namedef{subjclassname@2020}{%
  \textup{2020} Mathematics Subject Classification}
\makeatother

\begin{document}
\title[On the Ergodicity of Interacting Particle Systems]{On the Ergodicity of Interacting Particle Systems under Number Rigidity}

\author[K.~Suzuki]{Kohei Suzuki}
\address{Department of Mathematical Science, Durham University, Science Laboratories, South Road, DH1 3LE, United Kingdom}
\email{kohei.suzuki@durham.ac.uk}
\thanks{\hspace{-5.5mm}  Department of Mathematical Science, Durham University, South Road, DH1 3LE, United Kingdom
\\
\hspace{2.0mm} E-mail: kohei.suzuki@durham.ac.uk
\vspace{1mm} 
\\
The author gratefully acknowledges funding by the Alexander von Humboldt Stiftung.}

\keywords{\vspace{2mm}Ergodicity, tail triviality, optimal transport, number rigidity}


\begin{abstract} 
In this paper, we provide relations among the following properties:
\begin{enumerate}[(a)]
\item the tail triviality of a probability measure $\mu$ on the configuration space $\dUpsilon$; 
\item the finiteness of the $L^2$-transportation-type distance $\bar{\mssd}_\U$;
\item the irreducibility of $\QP$-symmetric Dirichlet forms on $\dUpsilon$.
\end{enumerate}
 As an application, we obtain the ergodicity (i.e., the convergence to the equilibrium) of interacting infinite diffusions having logarithmic interaction arisen from determinantal/permanental point processes including $\mathrm{sine}_{2}$, $\mathrm{Airy}_{2}$, $\mathrm{Bessel}_{\alpha, 2}$ ($\alpha \ge 1$), and $\mathrm{Ginibre}$ point processes, in particular, the case of unlabelled Dyson Brownian motion is covered. For the proof, the number rigidity of point processes in the sense of Ghosh--Peres plays a key role. 
\end{abstract}

\maketitle

\section{Introduction}
The ergodicity (i.e., the convergence to the equilibrium)  of interacting particle systems is one of the significant hypothesis supporting the foundation of statistical physics. In this paper, we study the ergodicity  in terms of the optimal transportation theory and of the theory of point processes. 


\paragraph{Configuration spaces}
The configuration space $\dUpsilon=\U(X)$ over~a locally compact Polish space~$X$ is the set of all locally finite point measures on $X$:
\begin{equation*} 
\dUpsilon(X)\eqdef \set{\gamma=\sum_{i=1}^N \delta_{x_i}: x_i\in X\comma N \in \N_0\cup \set{+\infty}\comma \gamma(K)<\infty \quad K \Subset X }\fstop
\end{equation*}
The space~$\dUpsilon$ is endowed with the \emph{vague topology}~$\T_\mrmv$ by the duality of compactly supported continuous functions on~$X$, and with a Borel probability measure~$\QP$,  understood as the law of a proper point process on~$X$. 

\paragraph{Interacting diffusions}A system of interacting {\it many} diffusions on the base space $X$ can be thought of as a {\it single} diffusion on $\dUpsilon$, provided the system does not condense too much by itself in the sense that every compact set in $X$ contains only finitely many particles throughout the time evolutions. 
There have been a large number of studies on a diffusion in~$\dUpsilon$, in particular, on a  system of infinite stochastic differential equations on~$\R^n$, written `formally' as
\begin{align} \label{1}
\diff X_t^k= - \frac{\beta}{2}\nabla \Phi(X_t^k)\diff t -  \frac{\beta}{2} \sum_{i \neq k} \nabla \Psi(X_t^k, X_t^i) \diff t + \diff B^k_t, \quad k \in \N \comma  
\end{align}
 whereby $\Phi$ is a free potential,  $\Psi$ is an interaction potential between particles, $\beta>0$ is a constant called {\it inverse temperature},  and $\seq{B^k}_{k \in \N}$ are independent Brownian motions on~$\R^n$. One approach addressing a solution to \eqref{1} is to construct a $\QP$-symmetric Dirichlet form~$\ttonde{\EE{\dUpsilon}{\QP}, \mathcal F^{\U, \QP}}$ on $L^2(\U, \QP)$, where $\QP$ is a (quasi-) Gibbs measure corresponding to the potentials $\Phi$ and $\Psi$, see, e.g., \cite{AlbKonRoe98b, Yos96} for Ruelle class potentials; \cite{Spo87, Osa96, Yoo05, Osa13, OsaTan14, HonOsa15, LzDSSuz21, Suz22b} for more general interactions including logarithmic potentials. The other approaches to tackle~\eqref{1} have been also studied such as the construction of time correlation functions in~\cite{Dys62, NagFor98, KatTan10}; the construction of the unique strong solution to  \eqref{1} in the case of the Dyson models in~\cite{Tsa16}.  We refer the readers to \cite{Roe09, Osa19} and also \cite[\S 1.6]{LzDSSuz21} for more complete accounts and references. 

\paragraph{Ergodicity}{The convergence to the equilibrium measure $\QP$} regarding solutions to \eqref{1} is characterised as {\it the ergodicity} of  the $L^2(\QP)$-semigroup $\{S^{\dUpsilon, \QP}_t\}$ corresponding to $\ttonde{\EE{\dUpsilon}{\QP}, \mathcal F^{\U, \QP}}$, which is defined as
\begin{align*} 
\int_{\dUpsilon} \biggl( S^{\dUpsilon, \QP}_t u - \int_{\dUpsilon} u \diff \QP \biggr)^2 \diff \QP \xrightarrow{t \to \infty} 0, \quad u \in L^2(\QP) \fstop
\end{align*}
An equivalent characterisation  is  the {\it irreducibility} of $\ttonde{\EE{\dUpsilon}{\QP}, \mathcal F^{\U, \QP}}$, i.e.,  
$$\EE{\dUpsilon}{\QP}(u)=0 \quad \implies  \quad u \equiv {\rm const.} \quad \text{$\QP$-a.e.} \fstop$$ 
Up to now, there were only few known examples, where one could show the ergodicity of $\{S^{\dUpsilon, \QP}_t\}$ in the case of infinite particle diffusions: one is a a class of Ruelle-type Gibbs measures with a {\it compactly supported interaction potential} and a {\it small activity constant~$z$} (\cite[Cor.\ 6.2]{AlbKonRoe98b}); the other is a {\it labelled} particle system corresponding to the $\mathrm{sine}_2$ process, which has been recently addressed in \cite{OsaTub21} by relying upon the arguments of strong solutions to~\eqref{1}  developed  in~\cite{OsaTan20}.


\paragraph{Optimal transport theory on $\dUpsilon$}
If the base space~$X$ is equipped with a metric~$\mssd$, the configuration space~$\dUpsilon$ is equipped with the \emph{$L^2$-transportation} (called also: \emph{$L^2$-Wasserstein}, or $L^2$-{\it Monge--Kantorovich--Rubinstein}) {distance}
\begin{align*}
\mssd_\dUpsilon(\gamma, \eta) \eqdef \inf \tonde{\int_{X^\tym{2}} \mssd^2(x,y)\diff\cpl(x,y)}^{1/2}\comma 
\end{align*}
where the infimum is taken over all measures~$\cpl$ on~$X^\tym{2}$ with marginals~$\gamma$ and~$\eta$.
As opposed to the case of the space of probability measures having finite second moment (i.e., the $L^2$-Wasserstein space), the function $\mssd_{\dUpsilon}$ cannot be a distance function because $\mssd_\dUpsilon$ may attain the value~$+\infty$ (e.g., when the total masses of $\gamma$ and $\eta$ are different, or the tails of $\gamma$ and $\eta$ are not close enough), and this {\it often} happens in the sense that this occurs on sets of positive measure for any reasonable choice of a reference measure on~$\dUpsilon$. It is, therefore, called \emph{extended} distance. In this article, we use a variant of $\mssd_\U$ defined as 
\begin{align*}
\bar{\mssd}_\U(\gamma, \eta):=
\begin{cases}
\mssd_\U(\gamma, \eta) \quad &\text{if $\gamma_{E^c}=\eta_{E^c}$ for some bounded set~$E$\ ,}
\\
+\infty \quad  &\text{otherwise} \fstop
\end{cases}
\end{align*}
Recent studies have revealed that the $L^2$-transportation distance is the right object to describe geometry, analysis and stochastic analysis in~$\U$ such as the curvature bounds on~$\dUpsilon$~(\cite{ErbHue15, LzDSSuz22, Suz22b}), the consistency between metric measure geometry and Dirichlet forms (\cite{RoeSch99, LzDSSuz21}, characterisations of BV functions and sets of finite perimeters on~$\dUpsilon$ (\cite{BruSuz21}) and the integral Varadhan short-time asymptotic~(\cite{Zha01, LzDSSuz22}).

\paragraph{Theory of point processes}A probability measure $\mu$ on $\dUpsilon$ is said to be {\it tail trivial} \ref{ass:T} if (see Dfn.~\ref{defn: TT}) 
$$\text{$\mu(A) \in \{0, 1\}\quad $ whenever $\quad A$ is a set in the tail $\sigma$-algebra} \fstop$$
The tail triviality has been originally discussed in relation to {\it phase transition} of Gibbs states (i.e., non-uniqueness of Gibbs measures with a given potential) and  it is equivalent to the {\it extremality} in the convex set of Gibbs measures with a given potential (see \cite[Cor.\ 7.4]{Geo11}). The tail triviality has been extended also to determinantal/permanental point processes by \cite{Lyo03} and \cite{ShiTak03b} independently. Since then, it has been further developed for a wider class of determinantal/permantental processes both in the continuous and discrete settings by various studies, see Example~\ref{exa: TT}. 
A probability measure $\QP$ on~$\dUpsilon$ is said to be {\it number rigid} (Assumption~\ref{ass:Rig}) if the following holds $\QP$-almost surely for every bounded Borel set $E$:
$$\gamma\mrestr{E^c} = \eta\mrestr{E^c} \quad \implies \quad \gamma(E) = \eta(E)\fstop$$
Namely, if two configuration $\gamma$ and $\eta$ coincide outside $E$, then the numbers of particles inside $E$ for $\gamma$ and $\eta$ coincide. 
The study of this remarkable spatial correlation phenomenon has been initiated by \cite{Gho12,  Gho15, GhoPer17} for $\mathrm{sine}_2$, $\mathrm{Ginibre}$ and $\mathrm{GAF}$ point processes and it has been further developed for other point processes, see Example~\ref{exa: R} for further references. 

\paragraph{Setting}In this article, we work in the following setting. 
Let $X=\R^n$ be the $n$-dimensional Euclidean space and $\mssd$ be the Euclidean distance on $\R^n$. Let $\seq{B_r}_{r \in \N}$ be a monotone increasing sequence of convex compact domains covering~$\R^n$ and $\mssm_r$ be the Lebesgue measure restricted on $B_r$. For  $E \subset \R^n$, define the projection $\pr_{E}: \U \ni \gamma \mapsto \gamma_{E}:=\gamma\mrestr{E}$ by the restriction of $\gamma$ on $E$. For a Borel probability measure $\QP$ on $\U$, define $\QP(\cdot\ |\ \cdot_{B_r^c}=\eta_{B_r^c})$ to be the regular conditional probability measure with respect to the $\sigma$-algebra $\sigma(\pr_{B_r^c})$ conditioned to be $\eta \in \U$. Define the push-forwarded measure and its restriction on $\U^k(B_r):=\{\gamma \in \U(B_r): \gamma(B_r)=k\}$ by
$$\QP_r^\eta:=(\pr_{B_r})_\#\QP(\cdot\ |\ \cdot_{B_r^c}=\eta_{B_r^c}) \comma \quad  \QP_r^{\eta, k}:=\QP_r^\eta\mrestr{{\U^k(B_r)}}\fstop$$ 
We denote by $\pi_{\mssm_r}$ the Poisson measure on $\U(B_r)$ with intensity $\mssm_r$ and by $\pi_{\mssm_r}^k$ the restriction on $\U^k(B_r)$. 
Let $\cdc^{\U(B_r)}$ be the square field on $\U(B_r)$ defined as 
$$\cdc^{\U(B_r)}(u):=\sum_{k=0}^\infty\cdc^{\U^k(B_r)}(u):=\sum_{k=0}^\infty\Bigl|\nabla^{\odot k}u|_{\U^k(B_r)}\Bigr|^2 \comma
$$
where $\nabla^{\odot k}$ is the symmetric product of the gradient operator~$\nabla$ on $\R^n$.
\smallskip
\begin{asm*}
We say that $\QP$ satisfies 
\begin{itemize} 
\item  {\it strong conditional absolute continuity} \ref{ass:SCE} if 
$$\QP_r^{\eta, k} \sim \pi_{\mssm_r}^k \comma \quad k \in  \mathcal K_r^\eta:=\{k \in \N_0: \QP_r^\eta(\U^k(B_r))>0\}  \quad r \in \N \quad \text{$\QP$-a.e.~$\eta$}\ ;$$
\item  {\it conditional closability} \ref{ass:ConditionalClos} if the form 
 $$\EE{\dUpsilon(B_r)}{\QP^{\eta}_{r}}(u) = \int_{\dUpsilon(B_r)} \cdc^{\dUpsilon(B_r)} (u) \diff\QP^{\eta}_{r}$$ is $L^2(\QP_r^\eta)$-closable on a certain core (see Dfn.~\ref{ass:ConditionalClosability})  for $\QP$-a.e.\ $\eta$ and every $r \in \N$. We denote its closure by $\dom{\EE{\dUpsilon(B_r)}{\QP^{\eta}_{r}}}$;
 \item {\it conditional irreducibility} \ref{ass:ConditionalErg} if 
 $$\E^{\U, \QP_r^\eta}(u) =0 \comma \quad u \in \dom{\EE{\dUpsilon(B_r)}{\QP^{\eta}_{r}}}\quad \implies  \quad u|_{\U^k(B_r)} \equiv C_r^{\eta, k} \quad \QP_r^{\eta,k} \text{-a.e.} \comma$$
 for $\QP$-a.e.~$\eta$, $r \in \N$, $k \in \mathcal K_r^\eta$, where $C_r^{\eta, k}$ is a constant depending on $r, \eta, k$. 
 \end{itemize}
  Under \ref{ass:SCE} and \ref{ass:ConditionalClos}, we construct a Dirichlet form~$\ttonde{\EE{\dUpsilon}{\QP},\dom{\EE{\dUpsilon}{\QP}}}$ in Prop.~\ref{p:DF}. Let $\mathcal F^{\U, \QP} \subset \dom{\EE{\dUpsilon}{\QP}}$ be any closed subspace satisfying the Markovian property~(see \eqref{e:Markov}), called {\it Markovian} subspace. 
We say that the form $(\E^{\U, \QP}, \mathcal F^{\U, \QP})$ satisfies
 \begin{itemize}  
 \item {\it Rademacher-type property} \ref{p:Rad} if 
\begin{align*}
\Lip_b(\bar{\mssd}_\U, \QP) \subset \mathcal F^{\U, \QP}\comma \quad \cdc^\U(u) \le \Lip_{\bar{\mssd}_\U}(u)^2 \comma
\end{align*}
where $\Lip_b(\bar{\mssd}_\U, \QP)$ is the space of bounded $\bar{\mssd}_\U$-Lipschitz $\QP$-measurable functions on~$\U$;
\item  {\it quasi-regularity} \ref{ass:QR} if 
$$\text{$(\E^{\U, \QP}, \mathcal F^{\U, \QP})$ is quasi-regular in $(\U, \tau_\mrmv)$}\comma$$
see \S\ref{subsec:D} for the definition of the quasi-regularity. 
\end{itemize}

 \end{asm*}
\paragraph{Main result}We define the following function associated with the $L^2$-transportation-type  distance~$\bar{\mssd}_{\dUpsilon}$:
$$\bar{\mssd}^\QP_\dUpsilon(\Xi, \Lambda):=\QP\text{-}\essinf_{\gamma \in \Xi}\inf_{\eta \in \Lambda}\bar{\mssd}_\dUpsilon(\gamma, \eta) \quad \Xi, \Lambda \subset \dUpsilon \fstop$$
We now state the main theorem, where we provide relations among the following three properties:
\begin{enumerate}[$(a)$]
\item \label{c:TT} $\mu$ is tail trivial~\ref{ass:T};
 \item \label{c:FD} $\bar{\mssd}_{\dUpsilon}^\QP(A, B)<\infty$ whenever $A$ is $\QP$-measurable,  $B$ is Borel and $\mu(A) \mu(B)>0$;
\item \label{c:IR}  $(\EE{\dUpsilon}{\QP}, \mathcal F^{\U, \QP})$ is irreducible.
\end{enumerate}
\begin{theorem}[Thm.~\ref{thm: Equiv}] \label{t:mthmI}
Let~$\QP$ be a Borel probability measure on~$\U$. 
Then,  
\begin{itemize}
\item \ref{c:FD}  $\implies$ \ref{c:TT};
\item if \ref{ass:Rig} holds,  then \ref{c:TT} $\implies$ \ref{c:FD}.
\end{itemize}
Suppose that $\QP$ satisfies~\ref{ass:SCE} and~\ref{ass:ConditionalClos},  and $\mathcal F^{\U, \QP} \subset  \dom{\EE{\dUpsilon}{\QP}}$ is any closed Markovian subspace. Then the following hold.
\begin{itemize}
\item if  \ref{ass:ConditionalErg}, \ref{ass:QR}  and \ref{ass:Rig} hold,  then \ref{c:FD} $\implies$ \ref{c:IR};
\item if \ref{p:Rad} holds,  then \ref{c:IR} $\implies$ \ref{c:FD}.
\end{itemize}
\end{theorem}

We therefore have the following relation between the tail triviality and the irreducibility.
\begin{corollary}[Tail triviality \& Irreducibility, {Cor.~\ref{cor: Equiv}}] \label{c:COI}
Let~$\QP$ be a Borel probability measure on~$\U$ satisfying \ref{ass:SCE}, ~\ref{ass:ConditionalClos}, and let $\mathcal F^{\U, \QP} \subset  \dom{\EE{\dUpsilon}{\QP}}$ be any closed Markovian subspace. Then the following hold.
\begin{itemize}
\item If  \ref{ass:ConditionalErg}, \ref{ass:QR}  and \ref{ass:Rig}  hold, then 
$$ \text{$\QP$ is tail trivial} \quad \implies \quad \text{$(\EE{\dUpsilon}{\QP}, \mathcal F^{\U, \QP})$ is irreducible} \;$$
\item If \ref{p:Rad} holds for $\mathcal F^{\U, \QP}$,
$$\quad \text{$(\EE{\dUpsilon}{\QP}, \mathcal F^{\U, \QP})$ is irreducible}\quad \implies   \quad  \text{$\QP$ is tail trivial} \fstop$$
\end{itemize}
\end{corollary}

 \paragraph{Applications}The first application of Theorem~\ref{t:mthmI} as well as Corollary~\ref{c:COI} is to considerably enlarge the list of (long-range) interactions for which one can prove the ergodicity of infinite particle systems. As an illustration, we will prove in \S \ref{sec: Exa} that $(\EE{\dUpsilon}{\QP}, \mathcal F^{\U, \QP})$ is irreducible (i.e., $\{S^{\dUpsilon, \QP}_t\}$ is ergodic) for all the measures~$\QP$ belonging to~$\mathrm{sine}_{2}$, $\mathrm{Airy}_{2}$, $\mathrm{Bessel}_{\alpha, 2}$ ($\alpha \ge 1$), and $\mathrm{Ginibre}$ point processes. In particular, the semigroup $\{S^{\dUpsilon, \QP}_t\}$ associated with {\it unlabelled Dyson Brownian motion} is covered. 
 


The second application is to show the finiteness of the $L^2$-transportation distance $\mssd_{\dUpsilon}(A, B)$ as well as $\bar{\mssd}_{\dUpsilon}(A, B)$ between sets~$A, B \subset \dUpsilon$. As both $\mssd_{\dUpsilon}$ and $\bar{\mssd}_{\dUpsilon}$ take value $+\infty$ on sets of positive measure, it is not straightforward to answer the following geometric question: 
\begin{align*}\tag*{{\sf (Q)}}
\text{when do $\mssd_{\dUpsilon}(A, B)$ and $\bar{\mssd}_{\dUpsilon}(A, B)$ return a finite value?}
\end{align*}
Theorem~\ref{t:mthmI} tells us the finiteness of~$\bar{\mssd}_{\dUpsilon}(A, B)$ (thus, also the finiteness of ${\mssd}_{\dUpsilon}(A, B)$ as $\mssd_\U \le \bar{\mssd}_{\dUpsilon}$ by definition) only by checking the positivity of measures~$\mu(A)\mu(B)>0$, due to the tail triviality \ref{ass:T} and the number rigidity \ref{ass:Rig} of $\QP$.



\paragraph{Comparisons with \cite{AlbKonRoe98b}}For a class of Gibbs measures or measures satisfying a certain integration-by-parts formula (denoted by {\rm (IbP1) and (IbP2)} in \cite[Thm.\ 6.2, 6.5]{AlbKonRoe98b}), relations between the ergodicity and the extremality of these measures have been studied. We compare our result with theirs in the following three points:
\begin{itemize}
\item {\bf Choice of a core}. They studied Dirichlet forms whose core is {\it cylinder functions} while our Dirichlet forms have a flexibility for the choice of a core, which for instance allows not only cylinder functions, but also local functions as well as Lipschitz functions. This broadens the score of applications significantly as cores of Dirichlet forms corresponding to long-range interactions constructed so far (e.g., \cite{Spo87, Osa96, Osa13, OsaTan14, HonOsa15, LzDSSuz21, Suz22b}) are covered by our setting, but not necessarily covered by the setting of cylinder functions.
\item {\bf Extremality vs.~ Tail-triviality}. They proved that the extremality of a class of Gibbs measures implies the ergodicity. The concept of the extremality is equivalent to the tail triviality when Gibbs measures  are considered (see~\cite[Cor.\ 7.4]{Geo11}).  However, the extremality is not necessarily defined  beyond Gibbs measures nor beyond measures satisfying {\rm (IbP1) and (IbP2)}, and many point processes coming from random matrix theory  are not always described as Gibbs measures nor {\rm (IbP1) and (IbP2)}, rather they are described by determinantal or permanental structures or by a scaling limit of eigenvalue distributions of random matrices.  
 In contrast, the tail triviality is a concept that can be defined for arbitrary point processes, because of which Theorem~\ref{t:mthmI} can be applied also to the latter cases.
\item {\bf Maximal domain vs.~Rademacher-type property}. They proved that the irreducibility of the maximal Dirichlet form implies the extremality of Gibbs measures, which corresponds to \ref{c:IR} $\implies$ \ref{c:TT} in Thereom~\ref{t:mthmI}. We however only assume the Rademacher-type property \ref{p:Rad} of our Dirichlet form, whose domain is in general smaller than the maximal form. As the irreducibility of a larger domain is a stronger statement,  Theorem~\ref{t:mthmI} proved the extremality of Gibbs measures (as well as the tail triviality of general measures) under a weaker assumption.  
 \end{itemize}
 \paragraph{Geometry and statistical physcis}We would like to draw the reader's attention that the relations between \ref{c:FD} and \ref{c:IR} in Theorem~\ref{t:mthmI} provides a relation between the ergodicity of interacting diffusion processes and a quantitative information of the optimal transport distance, where
the ergodicity is a {\it statistical-physical concept}, while the finiteness of the $L^2$-transportation distance between $\QP$-positive sets  is a {\it purely geometric concept} of the  extended metric measure space~$(\dUpsilon, \bar{\mssd}_{\dUpsilon}, \QP)$. 

\smallskip
We close this introduction by providing an outlook on further improvements. The number rigidity~\ref{ass:Rig} requires a strong spatial correlation to $\QP$, which is, however,  not a necessary condition for the ergodicity. Indeed, \cite[Thm.\ 4.3]{AlbKonRoe98} proved the ergodicity for the Poisson measures, which obviously do not posses the number rigidity~\ref{ass:Rig} since the laws of the Poisson point processes inside and outside bounded sets are independent. A challenging question is whether we can prove the ergodicity of Dirichlet forms for general tail trivial invariant measures without \ref{ass:Rig}. 

\paragraph{Organisation of the paper}In \S\ref{sec: Pre}, we introduce necessary concepts and recall results used for the arguments in later sections. In \S\ref{sec:CDF}, we construct Dirichlet forms on $\U$. In \S\ref{sec: Irr}, we prove the main results. 
 In \S\ref{sec: Ver}, we give sufficient conditions to verify the main assumptions of Theorem~\ref{t:mthmI}.  In~\S\ref{sec: Exa}, we confirm that Theorem~\ref{t:mthmI} can be applied to $\mathrm{sine}_{2}$, $\mathrm{Airy}_{2}$, $\mathrm{Bessel}_{\alpha, 2}$ ($\alpha \ge 1$), and $\mathrm{Ginibre}$ point processes.

\paragraph{Data Availability Statement}
Data sharing not applicable to this article as no datasets were generated or analysed during the current study.
\section{Preliminaries} \label{sec: Pre}

\subsection{Numbers, Tensors, Function Spaces}
We write $\N:=\{1, 2, 3, \ldots\}$, $\N_0=\{0, 1, 2, \ldots \}$, $\EN:=\N \cup \{+\infty\}$ and $\EN_0:=\N_0 \cup\{+\infty\}$. 
The uppercase letter $N$ is used for  $N \in \EN_0$, while the lowercase letter $n$ is used for $n \in \N_0$. 
We shall adhere to the following conventions:
\begin{itemize}
\item the superscript~${\square}^\tym{N}$ (the subscript~$\square_\tym{N}$) denotes $N$-fold \emph{product objects};

\item the superscript~${\square}^\otym{N}$ (the subscript~$\square_\otym{N}$) denotes $N$-fold \emph{tensor objects};

\item the superscript~${\square}^\osym{N}$ (the subscript~$\square_\osym{N}$) denotes $N$-fold \emph{symmetric tensor objects}.



\end{itemize}
Let~$(X, \tau)$ be a topological space with $\sigma$-finite Borel measure~$\nu$. Throughout this article, we shall use the following symbols and phrases:
\begin{enumerate}[$(a)$]
\item $L^p(\nu)$ $(1 \le p \le \infty)$ for the space of $\nu$-equivalence classes of real-valued functions~$u$ so that $|u|^p$ is $\nu$-integrable when $1 \le p<\infty$, and $u$ is $\nu$-essentially bounded when $p=\infty$; The $L^p(\nu)$-norm is denoted by $\|u\|^p_{L^p(\nu)}:=\|u\|^p_p:=\int_X |u|^p \diff \nu$ for $1 \le p <\infty$, and $\|u\|_{L^\infty(\nu)}:=\|u\|_\infty=\esssup_X u$;  When~$p=2$, the inner-product is denoted by $(u, v)_{L^2(\nu)}:=(u, v)_{2}:=\int_X uv \diff \nu$; 

\item  $L^p_s(\nu^{\otimes n}):=\{u \in L^p(\nu^{\otimes n}): u\ \text{is symmetric}\}$ where $u$ is said to be {\it symmetric} if $u(x_1, \ldots, x_k)=u(x_{\sigma(1)}, \ldots, x_{\sigma(k)})$ for every element $\sigma \in \mathfrak S(k)$ in the $k$-symmetric group;

\item $\mathscr B(X, \tau)$ for the Borel $\sigma$-algebra; $\mathscr B(X, \tau)^{\nu}$ for the completion of $\mathscr B(X, \tau)$ with respect to~$\nu$; $\mathscr B(X, \tau)^*$ be the universal $\sigma$-algebra, i.e., the intersection of $\mathscr B(X)^{\rho}$ among all Borel probability measures $\rho$ on $X$ (we do not specify the topology and simply write $\mathscr B(X), \mathscr B(X)^{\nu}, \mathscr B(X)^*$ where the topology is clear from the context); Measurable functions with respect to $\mathscr B(X)$, $\mathscr B(X)^{\nu}$, $\mathscr B(X)^*$ are called {\it Borel measurable, $\nu$-measurable, universally measurable} respectively. 
%
\item $C_b(X)$ for the space of $\T$-continuous bounded functions on~$X$; if $X$ is locally compact, $C_0(X)$ denotes the space of $\tau$-continuous and compactly supported  functions on $X$; $C_0^\infty(\R^n)$ for the space of compactly supported smooth functions on the $n$-dimensional Euclidean space~$\R^n$;
\item $\1_{A}$ for the indicator function on $A$, i.e., $\1_{A}(x)=1$ if and only if $x \in A$ and $\1_A(x)=0$ otherwise; $\delta_x$ for the Dirac measure at $x$, i.e., $\delta_x(A)=1$ if and only if $x \in A$ and $\delta_x(A)=0$ otherwise;
\item A sequence $(B_r)_{r \in \N}$ of subsets in $X$ is called {\it exhaustion} if $B_r \subset B_{r'}$ whenever $r \le r'$ and $\cup_{r \in \N} B_r = X$; If $B_r$ possesses a certain property $P$ for every $r \in \N$ (e.g., $B_r$ is compact, convex, or domain), we call it {\it P exhaustion} (e.g., compact exhaustion, compact convex exhaustion, domain exhaustion). 
\end{enumerate}

\subsection{Dirichlet form}\label{subsec:D}
We refer the reader to \cite{MaRoe90, BouHir91} for this subsection. Throughout this article, a Hilbert space always means a separable Hilbert space with inner product~$(\cdot, \cdot)_H$ taking value in $\R$. 

\paragraph{Dirichlet form}Given a bilinear form~$(Q,\dom{Q})$ on a Hilbert space~$H$, we write
\begin{align*}
Q(u)\eqdef Q(u,u) \comm \qquad Q_\alpha(u,v)\eqdef Q(u,v)+\alpha(u, v)_H\comm \alpha>0\fstop
\end{align*}
Let $(X, \Sigma, \nu)$ be a $\sigma$-finite measure space. A \emph{symmetric Dirichlet form on~$L^2(\nu)$} is a non-negative definite densely defined closed symmetric bilinear form~$(Q,\dom{Q})$ on~$L^2(\nu)$ satisfying the Markov property
\begin{align} \label{e:Markov}
u_0\eqdef 0\vee u \wedge 1\in \dom{Q} \quad \text{and} \quad Q(u_0)\leq Q(u)\comm \quad u\in\dom{Q}\fstop
\end{align}
Throughout this article, {\it Dirichlet form} always means {\it symmetric} Dirichlet form. A subspace $\mathcal F \subset \dom{Q}$ is called {\it Markovian subspace} if \eqref{e:Markov} holds for every $u \in \mathcal F$.
If not otherwise stated,~$\dom{Q}$ is always regarded as a Hilbert space with norm
$$\norm{\emparg}_{\dom{Q}}\eqdef Q_1(\emparg)^{1/2}:=\sqrt{Q(\emparg)+\norm{\emparg}_{L^2(\nu)}^2}\fstop $$
To distinguish Dirichlet forms defined in different base spaces with different reference measures, we write $Q^{X, \nu}$ to specify the base space $X$ and the reference measure $\nu$. We denote {\it the extended domain of $\dom{Q}$} by $\dom{Q}_e$ defined as 
\begin{align}\label{d:EDD}
\dom{Q}_e:=\{u: X \to \R:\ &\text{$\nu$-measurable,\ $|u|<\infty$ $\nu$-a.e.}
\\
&\text{$\exists \{u_n\}_{n \in \N}\subset \dom{Q}$ \ $Q$-Cauchy \ s.t.\ $u_n \to u$ \ $\nu$-a.e.}\} \notag
\end{align}

\paragraph{Square field}A Dirichlet form $(Q, \dom{Q})$ {\it admits square field $\cdc$} if there exists a dense subspace $H \subset \dom{Q} \cap L^\infty(\nu)$ so that for every~$u \in H$, there exists $v \in L^1(\nu)$ so that 
$$2Q(uh, u) -Q(h, u^2) = \int_X h v \diff \nu \quad h \in \dom{Q} \cap L^\infty(\nu) \fstop$$
This $v$ is denoted by $\Gamma(u)$. The square field $\Gamma$ can be uniquely extended as an operator on $\dom{Q}  \times \dom{Q} \to L^1(\nu)$ (\cite[Thm.\ I.4.1.3]{BouHir91}).

\paragraph{Resolvent,  semigroup and generator}We refer the reader to \cite[Chap.~I, Sec.~2]{MaRoe90} for this paragraph. Let $(Q, \dom{Q})$ be a symmetric closed form on a Hilbert space~$H$.  {\it The infinitesimal generator $(A, \dom{A})$} corresponding to $(Q, \dom{Q})$ is the unique densely defined closed operator on $H$ satisfying the following integration-by-parts formula:
$$-(u, Av)_{H}=Q(u, v) \quad \forall u \in \dom{Q},\ v\in \dom{A} \fstop$$
{\it The resolvent operator} $\{R_\alpha\}_{\alpha > 0}$ is the unique bounded linear operator on $H$ satisfying 
$$Q_\alpha(R_\alpha u, v) = (u, v)_{H} \quad \forall u \in H \quad v \in \dom{Q} \fstop$$
{\it The semigroup} $\{T_t\}_{t > 0}$ is the unique bounded linear operator on $H$ satisfying 
$$G_\alpha u = \int_{0}^\infty e^{-\alpha t} T_{t}u \diff t \quad u \in H\fstop$$

\paragraph{Locality}Let $(Q, \dom{Q})$ be a Dirihclet form on $L^2(\nu)$. It is called {\it local}  if for every $F, G \in C_c^\infty(\R)$ and  $u \in \dom{Q}$, 
$${\rm supp}[F] \cap {\rm supp}[G] = \emptyset \implies Q(F_0\circ u, G_0\circ u)=0 \comma$$
where $F_0(x):=F(x)-F(0)$ and $G_0(x):=G(x)-G(0)$ (see~\cite[Dfn.~5.1.2 in Chap.~I]{BouHir91}).

\paragraph{Quasi-notion}Let $(X, \tau)$ be a Polish space and $\nu$ be a $\sigma$-finite Borel measure on $X$ and $(Q, \dom{Q})$ be a Dirichlet form on $L^2(\nu)$.
For any~$A\in\mathscr B(X)$, define
$$\dom{Q}_A\eqdef \set{u\in \dom{Q}: u= 0 \ \text{$\nu$-a.e. on}\ X\setminus A}.$$
A sequence $\seq{A_n}_{n \in \N}\subset \mathscr B(X)$ is a \emph{Borel nest} if $\cup_{n \in \N} \dom{Q}_{A_n}$ is dense in~$\dom{Q}$.
A \emph{closed (resp.~compact) nest} is a Borel nest consisting of closed (resp.~compact) sets. 
A set~$N\subset X$ is \emph{exceptional} if there exists a closed nest~$\seq{F_n}_{n  \in \N}$ so that~$N\subset X\setminus \cup_n F_n$. It is a standard fact that any exceptional set $N$ is $\nu$-negligible (see, e.g., \cite[Exe.~2.3]{MaRoe90}). 
%
A property~$(p_x)$ depending on~$x\in X$ holds \emph{quasi-everywhere} (in short: q.e.) if there exists a polar set~$N$ so that~$(p_x)$ holds for every~$x\in X\setminus N$. For a closed nest $\seq{F_n}_{n \in \N}$, define 
$$C(\seq{F_n}_{n \in \N}):=\{u: A \to \R: \cup_{n \ge 1} F_n  \subset A \subset X,\ u|_{F_n} \ \text{is continuous for every $n \in \N$}\} \fstop$$
A function $u$ defined quasi-everywhere on $X$ is {\it quasi-continuous} if there exists a closed nest $\seq{F_n}_{n \in \N}$ so that $u \in C(\seq{F_n}_{n \in \N})$.  

A Dirichlet form $(Q, \dom{Q})$ on $L^2(\nu)$ is {\it quasi-regular} if the following conditions hold:
\begin{enumerate}[{\rm (QR1)}]
\item there exists a compact nest $(A_n)_{n \in \N}$; 
\item there exists a dense subspace $\mathcal D \subset \dom{Q}$ so that every $u \in \mathcal D$ has a quasi-continuous~$\nu$-version~$\tilde{u}$;
\item there exists $\{u_n: n \in \N\} \subset \dom{Q}$ and a polar set $N \subset X$ so that every $u_n$ has a quasi-continuous~$\nu$-version~$\tilde{u}_n$ and $\{\tilde{u}_n: n \in \N\}$ separates points in $X \setminus N$.
\end{enumerate}

\paragraph{Maximal function}Let $(Q, \dom{Q})$ be a local Dirichlet form  on $L^2(\nu)$ with $\mathbf 1 \in \dom{Q}$ having a square field $\cdc^{Q}$.  
Define 
$$\mathbb D_0:=\{u \in \dom{Q} \cap L^\infty(\nu): \cdc^{Q}(u) \le 1\} \fstop$$
By \cite[Thm.~1.2]{HinRam03}, for a $\nu$-positive measure set $A \subset X$, there exists a unique $\nu$-measurable function $\bar{\sf d}_{\nu, A}$ called {\it maximal function} satisfying 
\begin{enumerate}[(a)]
\item $\bar{\sf d}_{\nu, A} \wedge c \in \mathbb D_0$ for every $c \ge 0$;
\item $\bar{\sf d}_{\nu, A} =0$ $\nu$-a.e.~on~$A$;
\item $\bar{\sf d}_{\nu, A}$ is the largest function satisfying the previous two properties, i.e., if there exists a function $v$ satisfying (a) and (b), then 
\begin{align} \label{d:HRMF}
v \le \bar{\sf d}_{\nu, A} \quad \text{$\nu$-a.e.} \fstop
\end{align}
\end{enumerate}

\subsection{Extended metric space} 
Let~$X$ be any non-empty set. A function $\mssd\colon X^\tym{2}\rar [0,\infty]$ is called an \emph{extended distance} if it is symmetric, satisfying the triangle inequality and not vanishing outside the diagonal in~$X^{\tym{2}}$, i.e.~$\mssd(x,y)=0$ iff~$x=y$; a \emph{distance} if it is finite, i.e., $\mssd(x, y)<\infty$ for every $x, y \in X$. A space~$X$ equipped with an extended metric $\mssd$ is called {\it an extended metric space}~$(X, \mssd)$.
Let $\nu$ be a measure on $X$ on a $\sigma$-algebra $\Sigma$. Define 
\begin{align}\label{d:DFS}
\mssd(\cdot, B):=\inf_{y \in B}\mssd(\cdot, y) \comma\quad \mssd^\nu(A, B):=\nu\text{-}\essinf_{A}\inf_{y \in B}\mssd(\cdot, y)\comma \quad A, B \in \Sigma^\nu \comma
\end{align}
the latter of which is well-defined whenever $\inf_{y \in B}\mssd(\cdot, y)$ is $\nu$-measurable (i.e., $\Sigma^\nu$-measurable). 

\paragraph{Lipschitz algebra}A function~$f\colon X\rar \R$ is called {\it $\mssd$-Lipschitz} if there exists a constant~$L>0$ so that
\begin{align}\label{eq:Lipschitz}
\tabs{ u(x)-u(y)}\leq L\, \mssd(x,y) \comm \qquad x,y\in X \fstop
\end{align}
The smallest constant~$L$ satisfying~\eqref{eq:Lipschitz}  is called {\it the (global) Lipschitz constant of $u$}, denoted by~$\Lip_\mssd{(u)}$.
For any non-empty set~$A\subset X$, define~$\Lip(A,\mssd)$, resp.~$\bLip(A,\mssd)$ for the family of all $\mssd$-Lipschitz functions, resp.\ bounded $\mssd$-Lipschitz functions on~$A$.
For simplicity of notation, we omit specifying the base space~$X$ and simply write ~$\Lip(\mssd)\eqdef \Lip(X,\mssd)$, resp.\ $\bLip(\mssd)\eqdef \bLip(X,\mssd)$ if no confusion can occur. Define also $\Lip^\alpha(\mssd):=\{u \in \Lip(\mssd): \Lip_{\mssd}(u) \le \alpha\}$ and $\Lip^\alpha_b(\mssd):=\Lip^\alpha(\mssd) \cap \bLip(\mssd).$ For a measure $\nu$ on $X$ defined on a $\sigma$-algebra $\Sigma$ and a topology $\tau$ on $X$,  we define respectively
\begin{align*}
&\Lip(\mssd, \nu):=\{u \in \Lip(\mssd): \text{$u$ is $\nu$-measurable}\} \comma
\\
&\Lip(\mssd, \tau):=\{u \in \Lip(\mssd): \text{$u$ is $\tau$-continuous}\} \comma
\end{align*}
and we further define~$\Lip_b(\mssd, \nu)$, $\Lip^\alpha_b(\mssd, \nu)$, $\Lip_b(\mssd, \tau)$ and $\Lip^\alpha_b(\mssd, \tau)$ for the corresponding subspaces of $\nu$-measurable functions (resp.~$\tau$-continuous functions).

Let $\nu$ be a finite measure on $X$ and let $(Q, \dom{Q})$ be a local Dirichlet form on $L^2(\nu)$ having a square filed $\cdc^{Q}$. We say that {\it Rademacher-type property} holds for $\Lip_b({\mssd}, \nu)$ (resp.~$\Lip_b({\mssd}, \tau)$) if 
\begin{align}\tag*{$\mathsf{(Rad_{\mssd, \nu})}$}
\Lip_b({\mssd}, \nu) \subset \dom{Q}\comma \quad \cdc^Q(u) \le \Lip_{\mssd}(u)^2 \comma
\end{align}
\begin{align}\tag*{$\mathsf{(Rad_{\mssd, \tau})}$}
\Lip_b({\mssd}, \tau) \subset \dom{Q}\comma \quad \cdc^Q(u) \le \Lip_{\mssd}(u)^2 \comma \end{align}
respectively.

\subsection{Configuration space}
A \emph{configuration} on a locally compact Polish space~$X$ is an $\overline\N_0$-valued Radon measure~$\gamma$ on~$X$, which can be expressed by $\gamma = \sum_{i=1}^N \delta_{x_i}$ for $N \in \overline{\N}_0$, where $x_i \in X$ for every $i$ and  $\gamma \equiv 0$  when $N=0$.
%
The \emph{configuration space}~$\U=\dUpsilon(X)$ is the space of all configurations over~$X$. The space~$\dUpsilon$ is equipped with the vague topology~$\tau_\mrmv$, i.e., the topology generated by the duality of the space $C_0(X)$ of continuous functions with compact support. 
We write the {\it restriction}~$\gamma_A\eqdef \gamma\mrestr{A}$ for~$A \in \mathscr B(X)$ and the restriction map is denoted by 
\begin{align}\label{eq:ProjUpsilon}
\gamma\longmapsto \pr_A(\gamma):=\gamma_{A}\fstop
\end{align}
The $N$-particle configuration space is denoted by
\begin{equation*}
\begin{aligned}
\dUpsilon^N(X)\ \eqdef&\ \set{\gamma\in \dUpsilon: \gamma(X)=N}\comma
\end{aligned}
\quad N\in\overline\N_0 \fstop
\end{equation*}
Let $\mathfrak S_k$ be the $k$-symmetric group. It can be readily seen that the $k$-particle configuration space~$\U^k$ is isomorphic to the quotient space~$X^{\times k}/\mathfrak S_k$:
\begin{align} \label{e:STS}
\dUpsilon^k(X)\cong X^{\odot k}:=X^{\times k}/\mathfrak S_k \comma \quad k \in \N_0\fstop
\end{align}
The associated projection map from $X^{\times k}$ to the quotient space~$X^{\times k}/\mathfrak S_k$ is denoted by~$\quot_k$. 
For $\eta \in \U$ and~$E\in \mathscr B(X)$, we define
\begin{align} \label{e:CES}
\U_E^\eta:=\{\gamma \in \U: \gamma_{E^c}=\eta_{E^c}\} \fstop
\end{align}

\paragraph{Conditional probability}For a Borel probability measure~$\mu$ on~$\U$ and $E \in \mathscr B(X)$, 
$$\mu(\ \cdot \ | \ \pr_{E^c}(\cdot)=\eta_{E^c})$$ denotes the regular conditional probability of $\mu$ conditioned to be $\eta \in \U$ with respect to the $\sigma$-algebra generated by the projection map $\gamma \in \U \mapsto \pr_{E}(\gamma)=\gamma_{E} \in \U(E)$ (see e.g., \cite[Dfn.~3.32]{LzDSSuz21}).  Let~$\mu_{E}^\eta$ be the probability measure on~$\U(E)$ defined as 
\begin{align} \label{d:CPB}
\mu_{E}^\eta:=(\pr_{E})_\#\mu(\ \cdot \ | \ \pr_{E^c}(\cdot)=\eta_{E^c}) \comma
\end{align}
and its restriction on~the $k$-particle configuration space~$\U^k(E)$ is denoted by~$\mu_{E}^{\eta, k}:=\mu_{E}^\eta\mrestr{\U^k(E)}$.
\begin{rem}The conditional probability~$\mu(\ \cdot \ | \ \pr_{E^c}(\cdot)=\eta_{E^c})$ is a probability measure on~$\U$ whose support is contained in~$\U_E^\eta$ while $\mu_E^\eta$ is a probability measure on~$\U(E)$. We may identify the two of them without loss of information in the sense that  
\begin{align} \label{r:BMP}
\pr_{E}: \U_E^\eta \to \U(E) \ \text{is  a bi-measure-preserving bijection} \fstop
\end{align}
Namely, the projection map~$\pr_{E}$ is bijective with the inverse map $\pr_{E}^{-1}$ defined as~$\pr^{-1}_{E}(\gamma):=\gamma+\eta$, and both $\pr_{E}$ and~$\pr_{E}^{-1}$ are measure-preserving between the two measures~$\mu(\ \cdot \ | \ \pr_{E^c}(\cdot)=\eta_{E^c})$ and $\mu_{E}^\eta$.
\end{rem}  

For a measurable function $ u\colon \dUpsilon\to \R$, $E \in \mathscr B(X)$ and $\eta \in \dUpsilon$, we define 
 \begin{align} \label{e:SEF}
u_{E}^\eta(\gamma)\eqdef  u(\gamma+\eta_{E^\complement})  \qquad \gamma\in \dUpsilon(E) \fstop
 \end{align}
By  the property of the conditional probability, it is straightforward to see that for every $u \in L^1(\mu)$, 
\begin{align} \label{p:ConditionalIntegration}
\int_{\dUpsilon} u \diff\QP = \int_{\dUpsilon} \quadre{\int_{\dUpsilon(E)} u_{E}^\eta \diff \QP^\eta_E }\diff\QP(\eta) \fstop
\end{align}
See, e.g., \cite[Prop.\ 3.44]{LzDSSuz21}. For~$\Omega \in \mathscr B(\U)$, $E \in \mathscr B(X)$ and~$\eta \in \U$,  define~$\Omega_E^\eta \subset \U(E)$ as
 \begin{align} \label{e:SEF2}
 \Omega_E^\eta:=\{\gamma \in \U(E): \gamma+\eta_{E^c} \in \Omega\} \fstop
 \end{align}
 By applying the disintegration formula~\eqref{p:ConditionalIntegration} to $u=\1_{\Omega}$, we obtain
\begin{align} \label{p:ConditionalIntegration2}
\QP(\Omega)=\int_{\U} \QP_E^\eta(\Omega_E^\eta) \diff \QP(\eta) \fstop
\end{align}
%

\paragraph{Poisson measure}Let $(X, \tau, \nu)$ be a locally compact Polish space with Radon measure~$\nu$ satisfying~$\nu(X)<\infty$.  {\it The Poisson measure} $\pi_{\nu}$ on~$\U(X)$ with intensity~$\nu$ is defined in terms of the symmetric tensor measures $\{\nu^{\odot k}: k \in \N\}$ as follows:
\begin{align} \label{d:PS}
&\pi_{\nu}(\cdot):=e^{-\nu(X)}\sum_{k=1}^\infty \nu^{\odot k}\bigl(\cdot \cap {\U^k(X)}\bigr)=e^{-\nu(X)}\sum_{k=1}^\infty \frac{1}{k!}(\quot_k)_\#\nu^{\otimes k}\bigl(\cdot \cap {\U^k(X)} \bigr) \comma
\\
&\pi^k_{\nu}(\cdot):=\pi_{\nu}(\cdot)\mrestr{\U^{k}(X)} \fstop \notag
\end{align}
In the case that $\nu$ is $\sigma$-finite, take an exhaustion $\seq{B_r}_{r \in \N}$ so that $\nu(B_r)<\infty$ for every $r \in \N$. 
The \emph{Poisson} (\emph{random}) \emph{measure~$\PP_\nu$} with intensity~$\nu$ is the unique probability measure on~$\U$ satisfying
\begin{align}\label{eq:PoissonRestriction}
(\pr_{B_r})_\pfwd \PP_\nu=\PP_{\nu_{B_r}}\comma  \quad r\in \N \fstop
\end{align}
The measure $\PP_\nu$ does not depend on the choice of $\seq{B_r}_{r \in \N}$. 

\paragraph{$L^2$-transportation distance}Let $(X, \mssd)$ be a locally compact complete separable metric space. For~$i=1,2$ let~$\proj^i\colon X^{\times 2}\rar X$ denote the projection to the~$i^\text{th}$ coordinate for $i=1,2$. 
For~$\gamma,\eta\in \dUpsilon$, let~$\Cpl(\gamma,\eta)$ be the set of all couplings of~$\gamma$ and~$\eta$, i.e., 
\begin{align} \label{d:OTD}
\Cpl(\gamma,\eta)\eqdef \set{\cpl\in \Meas(X^{\tym{2}}) \colon \proj^1_\pfwd \cpl =\gamma \comma \proj^2_\pfwd \cpl=\eta} \fstop
\end{align}
Here $\Meas(X^{\tym{2}})$ denotes the space of all Radon measures on $X^{\tym{2}}$. 
The \emph{$L^2$-transportation}  \emph{extended distance} on~$\dUpsilon(X)$ is
\begin{align}\label{eq:d:W2Upsilon}
\mssd_{\dUpsilon}(\gamma,\eta)\eqdef \inf_{\cpl\in\Cpl(\gamma,\eta)} \paren{\int_{X^{\times 2}} \mssd^2(x,y) \diff\cpl(x,y)}^{1/2}\comma \qquad \inf{\emp}=+\infty \fstop
\end{align}
We refer the reader to e.g., \cite[Prop.~4.27, 4.29, Thm.~4.37, Prop.~5.12]{LzDSSuz21} and \cite[Lem.~4.1, 4.2]{RoeSch99} for details regarding the $L^2$-transportation extended distance $\mssd_{\dUpsilon}$ and examples of $\mssd_{\dUpsilon}$-Lipschitz functions. 
It is important to note that $\mssd_\dUpsilon$ is an {\it extended} distance, attaining the value~$+\infty$ and $\mssd_\U$ is lower semi-continuous with respect to the product vague topology $\tau_\mrmv^{\times 2}$ but never $\tau_\mrmv^{\times 2}$-continuous. 

We introduce a variant of the $L^2$-transportation extended distance, called \emph{$L^2$-transportation-type}  \emph{extended distance}~$\bar{\mssd}_\U$ defined as 
\begin{align} \label{eq:dW2L}
\bar{\mssd}_\U(\gamma, \eta):=
\begin{cases}
\mssd_\U(\gamma, \eta) \quad &\text{if $\gamma_{B_r^c}=\eta_{B_r^c}$ for some $r>0$\ ,}
\\
+\infty \quad  &\text{otherwise} \comma
\end{cases}
\end{align}
where $(B_r)_{r \in \N}$ is a compact exhaustion.
The definition~\eqref{eq:dW2L} does not depend on the choice of an exhaustion. 
By definition, $\mssd_\U \le \bar{\mssd}_\U$ on $\U$ and  $\mssd_\U = \bar{\mssd}_\U$ on $\U(E)$ for every compact subset~$E\subset X$. In particular, we have 
\begin{align} \label{e:LLR}
\Lip(\U, \mssd_\U) \subset \Lip(\U, \bar{\mssd}_\U) \comma \quad \Lip_{\bar{\mssd}_\U}(u) \le  \Lip_{\mssd_\U}(u)\comma  \quad u \in \Lip(\U, \mssd_\U)\fstop
\end{align}
It can be readily seen readily~that 
\begin{align} \label{e:LLR2}
\bar{\mssd}_\U(\gamma, \eta) <\infty \quad \iff \quad \gamma_{B_r^c}=\eta_{B_r^c} \comma \gamma(B_r)=\eta(B_r) \quad \text{for some $r>0$}\fstop
\end{align}
\begin{prop}The map $\bar{\mssd}_{\U}: \U^{\times 2} \to \R$ is $\mathscr B(\U^{\times 2}, \tau_\mrmv^{\times 2})$-measurable. 
\end{prop}
\begin{proof}
According to~\eqref{eq:dW2L}, we can write 
\begin{align} \label{d:MBL}
\bar{\mssd}_\U= \mssd_\U\1_{\Alpha} + \infty\1_{\Alpha^c} \comma
\end{align}
where 
\begin{align} \label{d:DDA}
\Alpha:=\{(\gamma, \eta) \in \U^{\times 2}: \ \exists r>0 \ \text{s.t.}\ \gamma_{B_r^c}=\eta_{B_r^c} \}=\cup_{r \in \N}\{(\gamma, \eta) \in \U^{\times 2}: \ \gamma_{B_r^c}=\eta_{B_r^c} \} \fstop
\end{align}
Let $\Alpha_r:=\{(\gamma, \eta) \in \U^{\times 2}: \ \gamma_{B_r^c}=\eta_{B_r^c} \}$, which is $\tau_\mrmv^{\times 2}$-closed. As $\Alpha$ is a countable union of closed sets, we obtain~$\Alpha \in \mathscr B(\U^{\times 2}, \tau_\mrmv^{\times 2})$. Noting that $\mssd_\U$ is $\tau_\mrmv^{\times 2}$-lower semi-continuous (\cite[(vi)~Lem.4.1]{RoeSch99}, see also~\cite[(vii) Prop.~4.27]{LzDSSuz21}), the function~$\mssd_\U$ is in particular $\mathscr B(\U^{\times 2}, \tau_\mrmv^{\times 2})$-measurable, thus, the expression~\eqref{d:MBL} concludes the $\mathscr B(\U^{\times 2}, \tau_\mrmv^{\times 2})$-measurablility of~$\bar{\mssd}_\U$.
\end{proof}
The following universal measurability of the distance function from a set will be used in Thm.~\ref{thm: Equiv}. 
\begin{prop}\label{p:UMD}Let $\Lambda \in \mathscr B(\U, \tau_\mrmv)$ and 
\begin{align} \label{d:SWD}
\bar{\mssd}_\U(\gamma, \Lambda):=\inf_{\eta \in \Lambda} \bar{\mssd}_\U(\gamma, \eta)\fstop
\end{align}
The map $\U \ni \gamma \mapsto \bar{\mssd}_\U(\gamma, \Lambda)$ is universally measurable (i.e., $\mathscr B(\U, \tau_\mrmv)^*$-measurable). 
\end{prop}
\begin{proof}
It suffices to show that every sub-level set~$\Lambda_r:=\{\gamma \in \U: \bar{\mssd}_\U(\gamma, \Lambda) \le r\}$ is universally measurable. 
Define $I: \U(X^{\times 2}) \to \R$ as  
$$\alpha \mapsto \int_{X^{\times 2}} \mssd(x, y) \diff \alpha(x, y) \fstop$$
The map~$I$ is lower semi-continuous in~$\U(X^{\times 2})$~(\cite[(i) Lem.~4.1]{RoeSch99}, see also~\cite[(ii) Prop.~4.27]{LzDSSuz21}). The following set~$\Beta_r$ is, therefore, closed in~$\U(X^{\times 2})$: 
$$\Beta_r:=\Bigl\{\alpha \in \U(X^{\times 2}): I(\alpha) \le r^2\Bigr\}\fstop$$
Noting~$\U^{\times 2} \subset \U(X^{\times 2})$ is a Borel subset, the Borel set~$\Alpha \in \mathscr B(\U^{\times 2})$ defined in~\eqref{d:DDA} can be thought of as a Borel set in~$\U(X^{\times 2}).$ 
Define $\tilde{\Beta}_r:=\Beta_r \cap \Alpha \in \mathscr B(\U(X^{\times 2}))$.
By~\eqref{eq:dW2L}, 
\begin{align} \label{e:LMDR}
\Lambda_r=\{\proj^1_\#\alpha: \alpha \in \tilde{\Beta}_r\comma \proj^2_\#\alpha \in \Lambda\}=\proj^1_\#\Bigl(\tilde{\Beta}_r \cap (\proj^2_\#)^{-1}(\Lambda)\Bigr) \comma
\end{align}
where $\proj^i$ has been defined just before~\eqref{d:OTD}. 
As the map~$\proj^{i}_\#: \U(X^{\times 2}) \to \U$  is continuous, the set~$\tilde{\Beta}_r \cap (\proj^2_\#)^{-1}(\Lambda)$ is a Borel set in~$\U(X^{\times 2})$. 
Noting the fact that a continuous image of a Borel set in a Polish space is Suslin  (e.g., \cite[Thm.~21.10]{Kec95}), we conclude by~\eqref{e:LMDR} that~$\Lambda_r$ is a Suslin set in $\U$, therefore, universally measurable (see, e.g., \cite[431B Corollary]{Fre01}). 
\end{proof}

We recall a lemma, which states that the operation $(\cdot)_{E}^\eta$ defined in~\eqref{e:SEF} maps from $\Lip(\U, \bar{\mssd}_\U)$ to $\Lip(\U(E), \mssd_\U)$ and contracts Lipschitz constants. 
\begin{lem}\label{l:SEF3}
Let $u \in \Lip(\U, \bar{\mssd}_\U)$ and $E \subset X$ be a Polish subset. Then, $u_{E}^\eta \in \Lip(\U(E), \mssd_\U)$ and 
\begin{align} \label{e:SEF3}
\Lip_{\mssd_\U}(u_E^\eta) \le \Lip_{\bar{\mssd}_\U}(u) \comma \quad \eta \in \U\fstop
\end{align}
\end{lem}
\begin{proof}
Let $\gamma, \zeta \in \U(E)$ and $\eta \in \U$. Then, 
\begin{align*}
|u_E^\eta(\gamma)-u_E^\eta(\zeta)|=|u(\gamma+\eta_{E^c})-u(\zeta+\eta_{E^c})| &\le \Lip_{\bar{\mssd}_\U}(u) \bar{\mssd}_\U(\gamma+\eta_{E^c}, \zeta+\eta_{E^c}) 
\\
&= \Lip_{\bar{\mssd}_\U}(u) \mssd_\U(\gamma, \zeta) \fstop
\end{align*}
The proof is completed. 
\end{proof}
\begin{rem} By the same proof,  one can replace $\bar{\mssd}_\U$ with $\mssd_\U$ in the statement of Lem.~\ref{l:SEF3} and obtain
\begin{align} \label{e:SEF3-2}
\Lip_{\mssd_\U}(u_E^\eta) \le \Lip_{{\mssd}_\U}(u) \comma \quad \eta \in \U\fstop
\end{align}
\end{rem}
\subsection{Tail triviality} \label{subsec: TT}
Let $\seq{B_r}_{r \in \N}$ be a compact exhaustion. 
Let $\sigma({\rm pr}_{B_r^c})$ denote the $\sigma$-algebra generated by the projection map $\U \ni \gamma \mapsto {\rm pr}_{B_r^c}(\gamma)=\gamma_{B_r^c} \in \U(B_r^c)$. 
We set $\mathscr T(\U):=\cap_{r \in \N}\sigma({\rm pr}_{B_r^c})$ and call it {\it tail $\sigma$-algebra}. 
By the definition of the tail $\sigma$-algebra~$\mathscr T(\dUpsilon)$, every set $\Xi \in \mathscr T(\dUpsilon)$ satisfies the following condition:
\begin{align} \label{eq: EXP0}
\Xi=\U(B_r) + {\rm pr}_{B_r^c}(\Xi), \quad  r \in \N\fstop
\end{align}
For a set $\Xi \subset \U$,  define $\mathcal T_{B_r}({\Xi}):=({\rm pr}_{B_r^c})^{-1}\circ {\rm pr}_{B_r^c}(\Xi)$. By definition, $\Xi \subset \mathcal T_{B_r}(\Xi)$, and $\mathcal T_{B_r}({\Xi}) \subset \mathcal T_{B_{r'}}({\Xi})$ whenever $r\le r'$. 
Define {\it the tail set of $\Xi$} by 
\begin{align} \label{eq: ts}
\mathcal T(\Xi):=\cup_{r \in \N} \mathcal T_{B_{r}}({\Xi}) \fstop
\end{align}
The tail set $\mathcal T(\Xi)$ of $\Xi$ does not depend on the choice of the exhaustion~$\seq{B_{r}}$. It can be readily shown that $\mathcal T(\Xi) \in \mathscr T(\U)$ and $\Xi \subset \mathcal T({\Xi})$. 
\begin{defs}[Tail triviality] \normalfont \label{defn: TT}
A Borel probability measure $\mu$ on $\U(X)$ is called {\it tail trivial} \ref{ass:T} if 
\begin{align}\tag*{$(\mathsf{T})_{\ref{defn: TT}}$}\label{ass:T}
\mu(\Xi)\in \{0, 1\} \quad  \text{whenever} \quad \Xi \in \mathscr T(\U) \fstop
\end{align}
\end{defs}

\begin{ese}\normalfont  \label{exa: TT}
The tail triviality has been verified for a wide class of point processes. 
\begin{itemize}
\item[(i)] (Determinantal point processes) Let $X$ be a locally compact Polish space. 
Then, all determinantal point processes whose kernel are locally trace-class positive contraction satisfy the tail triviality (see \cite[Theorem 2.1]{Lyo18} and \cite{BufQiuSha21,  OsaOsa18, ShiTak03b}). In particular, $\mathrm{sine}_2$,~$\mathrm{Bessel}_{\alpha,2}$, $\mathrm{Airy}_2$ and Ginibre point processes are tail trivial. 
\item[(ii)] (Extremal Gibbs measure)  A canonical Gibbs measure $\mu$ is tail  trivial iff $\mu$ is extremal (see \cite[Cor.\ 7.4]{Geo11}). In particular, Gibbs measures of the Ruelle type with sufficiently small activity constants are extremal (see \cite[Thm.\ 5.7]{Rue70}).
\end{itemize}
\end{ese}

\subsection{Number-rigidity} \label{subsec: Rig}
The following definition of the number rigidity on the configuration space $\U$ over a locally compact Polish space $X$ is an adaptation of the number rigidity introduced by Ghosh--Peres \cite{GhoPer17} originally in the setting of the configuration space over the complex plane. 
\begin{defs}[Number-rigidity: cf.~Ghosh--Peres {\cite[Thm.\ 1]{GhoPer17}}]\label{ass:Rigidity}
A Borel probability measure $\QP$ on $\U$ has {\it the number rigidity} (in short: \ref{ass:Rig}) if, for every bounded Borel set~$E\subset X$, there exists $\Omega \subset \dUpsilon$ so that $\QP(\Omega)=1$  and, for every $\gamma, \eta \in \Omega$
\begin{align*}\tag*{$(\mathsf{R})_{\ref{ass:Rigidity}}$}\label{ass:Rig}
\text{$\gamma_{E^c} = \eta_{E^c}$ implies $\gamma(E) = \eta(E)$}  \fstop
\end{align*}
\end{defs}
\begin{ese} \label{exa: R}
The number rigidity has been verified for a variety of point processes: Ginibre and GAF (\cite{GhoPer17}), $\mathrm{sine}_\beta$ (\cite[Thm.\ 4.2]{Gho15}, \cite{ChhNaj18}, \cite{DerHarLebMai20}), $\mathrm{Airy}$, ~$\mathrm{Bessel}$, and $\mathrm{Gamma}$ (\cite{Buf16}), and Pfaffian (\cite{BufNikQiu19}) point processes. We refer the readers also to the survey \cite{GhoLeb17}. 
\end{ese}


\section{Construction of Dirichlet forms}\label{sec:CDF}
In this section, we construct a Dirichlet form on $\U=\U(\R^n)$. Let $\seq{B_r}_{r \in \N}$ be a compact convex domain exhaustion in $\R^n$. 
We first construct a Dirichlet form on~$\U(B_r)$ called {\it conditioned form} with invariant measure~$\QP_{B_r}^\eta$. We then lift it onto $\U$, which is called {\it truncated form}, whose gradient operator is truncated on $B_r$. Finally we take the monotone limit of the truncated forms as $r\to\infty$ and construct the limit Dirichlet form on $\U$.   

\paragraph{Notation}Hereinafter, we use the following notation.
\begin{itemize}
\item $\mssm$, $\mssm_r$ for the Lebesgue measure on $\R^n$ and its restriction on $B_r$ respectively;
\item $\mssd(x, y):=|x-y|$ for  the Euclidean distance in $\R^n$;
\item $\QP_r^\eta:=\QP_{B_r}^\eta$ for a probability measure $\QP$ on $\U$, defined in~\eqref{d:CPB};
\item $u_r^{\eta}:=u_{B_r}^\eta$ for a function $u: \U \to \R$, defined in~\eqref{e:SEF}.
\end{itemize}
\subsection{Conditioned Dirichlet forms on $\dUpsilon(B_r)$}

Let $W^{1,2}_s(\mssm^{\otimes k}_r)$ be the space of $\mssm^{\otimes k}_r$-classes of~$(1,2)$-Sobolev and {\it symmetric} functions on the product space $B_r^{\times k}$, i.e., 
$$W^{1,2}_s(\mssm^{\otimes k}_r):=\biggl\{u \in L^2_s(\mssm^{\otimes k}_{r}): \int_{B_r^{\times k}} |\nabla^{\otimes k} u|^2 \diff \mssm^{\otimes k}_{r} <\infty \biggr\} \comma$$
where $\nabla^{\otimes k}$ denotes the weak derivative on $(\R^n)^{\times k}$: $\nabla^{\otimes k}u:=(\partial_1 u, \ldots, \partial_ku)$.   
As the space $W^{1,2}_s(\mssm^{\otimes k}_r)$ consists of symmetric functions, the projection $\quot_k: B_r^{\times k} \to \U^k(B_r) \cong B_r^{\times k} /\mathfrak S_k$ acts on $W^{1,2}_s(\mssm^{\otimes k}_r)$ and the resulting quotient space is denoted by $W^{1,2}(\mssm_r^{\odot k})$:
$$W^{1,2}(\mssm^{\odot k}_r):=\biggl\{u \in L^2(\mssm^{\odot k}_{r}): \int_{\U^k(B_r)} |\nabla^{\odot k} u|^2 \diff \mssm^{\odot k}_{r} <\infty \biggr\} \comma$$
where $\nabla^{\odot k}$ is the quotient operator of the weak gradient operator $\nabla^{\otimes k}$ through the projection $\quot_k$ and $\mssm_r^{\odot k}$ is the symmetric product measure defined as 
$$\mssm_r^{\odot k}:=\frac{1}{k!} (\quot_k)_\#\mssm_r^{\otimes k} \fstop$$
\begin{defs}[Conditional absolute continuity]\label{d:ConditionalAC}
A Borel probability measure~$\QP$ on $\dUpsilon$ 
is \emph{conditionally absolutely continuous} (\emph{to~$\PP_\mssm$}) if 
\begin{equation}\tag*{$(\mathsf{CAC})_{\ref{d:ConditionalAC}}$}
\label{ass:CE}
 \QP^{\eta, k}_{r} \ll \PP_{\mssm_{r}}\mrestr{\dUpsilon^{k}(B_{r})}  \quad r\in \N\comma \ k \in \N_0  \comma \  \QP\text{-a.e.~$\eta$}\fstop
\end{equation}
Let $\mathcal K_r^\eta:=\{k \in \N_0: \QP_r^\eta(\U^k(B_r))>0\}$. We say that $\QP$ satisfies \ref{ass:SCE} if 
\begin{equation}\tag*{$(\mathsf{CAC'})_{\ref{d:ConditionalAC}}$}
\label{ass:SCE}
 \QP^{\eta, k}_{r} \sim \PP_{\mssm_{r}}\mrestr{\dUpsilon^{k}(B_{r})}  \quad r \in \N\comma \ \QP\text{-a.e.~$\eta$} \comma\ k \in \mathcal K_r^\eta   \fstop
\end{equation}
\end{defs}

For $u, v: \U(B_r) \to \R$ satisfying $u|_{\U^{k}(B_r)}, v|_{\U^{k}(B_r)} \in W^{1,2}(\mssm_r^{\odot k})$ for every $k \in \N$, set
\begin{align} \label{d:SFD}
\cdc^{\dUpsilon(B_r)}(u, v)& := \sum_{k=0}^\infty \Bigl\langle \nabla^{\odot k} u|_{\U^{k}(B_r)},  \nabla^{\odot k} v|_{\U^{k}(B_r)}\Bigr\rangle \comma \quad \cdc^{\dUpsilon(B_r)}(u):=\cdc^{\dUpsilon(B_r)}(u, u) \fstop 
\end{align}
Let us define the following algebra 
$${\rm LIP}_b(\U(B_r), \mssd_\U):=\{u: \U(B_r) \to \R \ \text{bounded}: u|_{\U^k(B_r)} \in \Lip_b(\U^k(B_r), \mssd_\U)\comma k \in \N\}\fstop$$
Note that the Lipschitz constant $\Lip_{\mssd_\U}(u|_{\U^k(B_r)})$ may not be bounded in $k$ for $u \in {\rm LIP}_b(\U(B_r), \mssd_\U)$, thus 
$$\Lip_b(\U(B_r), \mssd_\U) \subsetneq {\rm LIP}_b(\U(B_r), \mssd_\U) \fstop$$
The quadratic functional associated with~$\QP^{\eta, k}_{r}$ is denoted by
\begin{subequations}\label{eq:VariousForms}
\begin{align}
\label{eq:VariousFormsB}
&\EE{\dUpsilon(B_{r})}{\QP^{\eta, k}_{r}}(u) \eqdef \int_{\dUpsilon(B_{r})} |\nabla^{\odot k}u |^2\diff\QP^{\eta, k}_{r} \comma
\\
\label{eq:VariousFormsC}
&\EE{\dUpsilon(B_{r})}{\QP^\eta_{r}}(u) \eqdef \int_{\U(B_r)} \cdc^{\dUpsilon(B_r)}(u) \diff \QP_r^\eta \comma  \quad u \in {\rm LIP}_b(\U(B_{r}), \mssd_{\U})\fstop
\end{align}
\end{subequations}

\begin{defs}[Conditional closability]\label{ass:ConditionalClosability}
Let~$\QP$ be a Borel probability measure on~$\dUpsilon$ satisfying~\ref{ass:CE}. 
%
We say that~$\QP$ satisfies the \emph{conditional closability}~\ref{ass:ConditionalClos} if the form
\begin{align}\tag*{$(\mathsf{CC})_{\ref{ass:ConditionalClosability}}$}\label{ass:ConditionalClos}
&\EE{\dUpsilon(B_{r})}{\QP^\eta_{r}}(u,v)=\int_{\dUpsilon(B_{r})} \cdc^{\dUpsilon(B_{r})}(u,v) \diff\QP^\eta_{r}\comma
\\
& u,v \in {\rm LIP}_b(\U(B_{r}), \mssd_{\U}) \cap \{u: \U(B_r) \to \R: \EE{\dUpsilon(B_{r})}{\QP^\eta_{r}}(u)<\infty\} \fstop \notag
\end{align}
is closable on~$L^2\ttonde{\dUpsilon(B_r),\QP^\eta_{r}}$ for every $r \in \N$ and $\QP$-a.e.~$\eta\in\dUpsilon$. 
\end{defs}
\begin{rem} \label{r:HMZ} We give two remarks on~\ref{ass:ConditionalClos}. 
\begin{enumerate}[(i)]
\item The Rademacher theorem on~convex domains in the Euclidean space implying
\begin{align} \label{e:RDMF}
&\Lip_b(\U(B_{r}), {\mssd}_{\U})|_{\U^k(B_r)} \subset W^{1,2}(\mssm^{\odot k}_{r})\comma
\\
& \Bigl|\nabla^{\odot k}u|_{\U^k(B_r)}\Bigr| \le \Lip_{\mssd_\U}(u)  \quad \text{on} \quad \U^k(B_r)\comma \quad k \in \N \comma \notag
\end{align}
the following bound follows:
\begin{align} \label{e:RDMF2}
\cdc^{\dUpsilon(B_r)}(u) \le  \Lip_{\mssd_\U}(u)^2\comma \quad u \in \Lip_b(\U(B_{r}), \mssd_{\U}) \comma
\end{align}
which shows
$$\Lip_b(\U(B_r), \mssd_\U) \subset  {\rm LIP}_b(\U(B_{r}), \mssd_{\U}) \cap \{u: \U(B_r) \to \R: \EE{\dUpsilon(B_{r})}{\QP^\eta_{r}}(u)<\infty\} \fstop$$
thus,  the form~\eqref{eq:VariousFormsC} is well-posed on~$\Lip_b(\U(B_r), \mssd_\U)$. 
\item A simple sufficient condition for \ref{ass:ConditionalClos} is 
$$\phi_{r}^{\eta, k}:=\frac{\diff \QP_{r}^\eta}{\diff \pi_{\mssm_{r}}}\Big|_{\U^{k}(B_{r})} \in C_b(\U^{k}(B_{r})) \quad r, k \in \N \fstop$$
In this case,  the closability of the form $\EE{\dUpsilon(B_{r})}{\QP^{\eta, k}_{r}}$  is the standard consequence of the Hamza-type argument by \cite{RoeWie85, Fuk97}, see for an accessible reference, e.g., \cite[pp.\ 44-45]{MaRoe90}. The closability of the form $\EE{\dUpsilon(B_{r})}{\QP^\eta_{r}}$ then follows as it is a countable sum of closable forms  $\EE{\dUpsilon(B_{r})}{\QP^{\eta, k}_{r}}$ along $k \in \N_0$ (see e.g., \cite[Prop.~3.7]{MaRoe90}). All examples we shall discuss in Section~\ref{sec: Exa} fall into this case. 
\end{enumerate}
\end{rem}
\begin{defs}[Conditioned form]\label{def:CNDF}
Under~\ref{ass:CE} and~\ref{ass:ConditionalClos}, 
the closure of \eqref{eq:VariousFormsC} is called {\it conditioned form} and denoted by
\begin{equation} \label{eq: condF}
\ttonde{\EE{\dUpsilon(B_r)}{\QP^\eta_{r}},\dom{\EE{\dUpsilon(B_r)}{\QP^\eta_{r}}}} \fstop
\end{equation}
The corresponding $L^2(\QP^\eta_r)$-resolvent operator and the $L^2(\QP^\eta_r)$-semigroup are denoted respectively by 
$$\bigl\{G_{\alpha}^{\U(B_r), \QP^\eta_r}\bigr\}_{\alpha > 0} \quad \text{and} \quad \bigl\{T^{\U(B_r), \QP^\eta_r}_t\bigr\}_{t > 0}\fstop$$ 
The square field~$\cdc^{\dUpsilon(B_r)}$ naturally extends to the domain $\dom{\EE{\dUpsilon(B_r)}{\QP^\eta_{r}}}$, which is denoted by the same symbol~$\cdc^{\dUpsilon(B_r)}$.
\smallskip
\end{defs}

\subsection{Truncated Dirichlet forms}
In this subsection, we construct the truncated Dirichlet form on~$\U$.
 We start this section by giving an operator mapping functions on $\U$ to functions on~$\R^n$.
\begin{defs}[{\cite[Lem.~1.2]{MaRoe00}, see also \cite[Lem.~2.16]{LzDSSuz21}}]\ 
For~$u: \U \to \R$, define~$\mathcal U_{\gamma, x}(u): \R^n \to \R$ by 
\begin{align} \label{d:UO}
\mathcal U_{\gamma, x}(u)(y):=u\ttonde{\car_{X\setminus\set{x}}\cdot\gamma + \delta_y}-u\ttonde{\car_{X\setminus\set{x}}\cdot\gamma} \comma \quad \gamma \in \U,\quad x \in \gamma \fstop
\end{align}
\end{defs}

We now define a square field operator on $\U$ truncated up to particles inside $B_r$.
\begin{defs}[Truncated square field on $\U$]\label{d:DT}
 The following operator is called {\it the truncated square field} $\Gamma^\U_r$ whenever $\nabla \mathcal U_{\gamma, x}(u)|_{B_r}$ makes sense $\mssm_r$-a.e.~ for $u:\U \to \R$:
\begin{equation}\label{eq:d:LiftCdCRep}
\begin{gathered}
\cdc^{\dUpsilon}_r(u)(\gamma):= \sum_{x\in \gamma_{B_r}} |\nabla \mathcal U_{\gamma, x}(u)|^2(x)\fstop
\end{gathered}
\end{equation}
Thanks to~Lem.~\ref{l:WDG}, 
Formula~\eqref{eq:d:LiftCdCRep}  is well-defined for $\QP$-a.e.~$\gamma$. Indeed, 
as the weak gradient~$\nabla \mathcal U_{\gamma, x}(u)$ is well-defined pointwise on a measurable set $\Sigma \subset B_r$ with $\mssm_r(\Sigma^c)=0$, by applying  Lem.~\ref{l:WDG}, Formula~\eqref{eq:d:LiftCdCRep} is well-defined on a set~$\Omega(r)$ of $\QP$-full measure. 
\end{defs} 

Based on the truncated square field $\cdc^\U_r$, we introduce the truncated form on $\U$ defined on a certain core. 
\begin{defs}[Core]\label{d:core}Let $\{\mathcal C_r\}_{r \in \N}$ be a sequence of  algebras of $\mu$-classes of measurable functions so that $\mathcal C_r \supset \mathcal C_{r'}$ for $r \le r'$ and the following hold for every $r \in \N$:
\begin{enumerate}[$(a)$]
\item\label{i:d:core1} $\mathcal C_r \subset L^\infty(\mu)$ and $\mathcal C_r \subset L^2(\QP)$ is dense;  
\item\label{i:d:core2} 
$\cdc^{\U}_r(u)$ is well-defined $\QP$-a.e.~for every~$u \in \mathcal C_r$;

\item\label{i:d:core3}   the following integral is well-defined and finite for every $u \in \mathcal C_r $:
\begin{align} \label{eq:VariousFormsA} 
 \text{$u_{r}^\eta \in \dom{\E^{\U(B_r), \QP_r^\eta}}$ $\quad \QP$-a.e.~$\eta$} \comma \qquad \E^{\U, \mu}_r(u):= \int_\dUpsilon  \E^{\U(B_r), \mu_r^\eta}(u_{r}^\eta) \diff\QP(\eta) <\infty \comma
\end{align} 
and $(\E^{\U, \mu}_r, \mathcal C_r)$ is Markovian.

%
%
\end{enumerate}
\end{defs}  
\begin{ese}\label{ese:CC}
We have several choices of $\{\mathcal C_r\}_{r \in \N}$. In each of the following examples, we take a certain common core $\mathcal C$ and take $\mathcal C_r = \mathcal C$ for every $r>0$. 
\begin{enumerate}[(a)]
\item {\it Cylinder functions}. Take $\mathcal C_r=\mathcal C=\Cyl{D}$ for every~$r \in \N$, where $\Cyl{D}$ is the space of cylinder functions defined as 
\begin{align} \label{defn: cyl}
\CYL\eqdef \set{\begin{matrix}  u\colon \dUpsilon\rar \R :  u=F\bigl(\gamma(f_1), \gamma,(f_2), \ldots, \gamma(f_k)\bigr) \comma  F\in \mcC^\infty_b(\R^k)\comma \\  f_1,\dotsc,  f_k\in  C_0^\infty(\R^n)\comma\quad k\in \N_0 \end{matrix}}\comma
\end{align}
where $\gamma(f):=\int_{\R^n} f \diff \gamma$.
We say that~$\QP$ satisfies \ref{ass:Mmu} if the intensity measure $\mssm_\QP$ is locally finite intensity, viz.\
\begin{equation}\tag*{$(\mssm_\QP)$}\label{ass:Mmu}
\mssm_\QP E:=\int_{\U} \gamma(E) \diff \QP(\gamma)<\infty \comma \quad E \subset \R^n\ \text{compact}\fstop
\end{equation}
Under~\ref{ass:Mmu},~\ref{ass:CE} and~\ref{ass:ConditionalClos},  all the conditions of Dfn.~\ref{d:core} are satisfied (see~\cite[Lem.~2.15, Prop.~3.45, Thm.~3.48]{LzDSSuz21}).
\item {\it Lipschitz functions}. Take $\mathcal C_r =\mathcal C$  for every $r \in \N$, where $\mathcal C$ is equal to either
\begin{align} \label{c:COC}
\Lip_b(\bar{\mssd}_\U, \QP)\comma   \Lip_b(\mssd_\U, \QP)\comma   \Lip_b(\bar{\mssd}_\U, \QP) \cap C_b(\tau_\mrmv)\comma \text{or}  \ \   \Lip_b(\mssd_\U, \QP) \cap C_b(\tau_\mrmv) \fstop
\end{align} 
As $\Lip_b(\mssd_\U, \QP) \subset L^2(\mu)$ is dense (e.g., \cite[Prop.\ 4.1]{AmbGigSav14}) and $\Lip_b({\mssd}_\U, \QP) \subset \Lip_b(\bar{\mssd}_\U, \QP)$ by~\eqref{e:LLR}, $\Lip_b(\bar{\mssd}_\U, \QP) \subset L^2(\QP)$ is dense as well. The density of $\Lip_b({\mssd}_\U, \QP) \cap C_b(\tau_\mrmv)$ follows e.g., by~\cite[Lem.~2.27]{Sav19} combined with the fact \cite[Prop.~4.30]{LzDSSuz21} that $(\U, \mssd_\U, \tau_\mrmv)$ is an extended metric-topological space. This therefore implies the density of~$\Lip_b(\bar{\mssd}_\U, \QP) \cap C_b(\tau_\mrmv)$ as well. 
Thanks to the Lipschitz contraction property of $(\cdot)_r^\eta$ by Lem.~\ref{l:SEF3} and of $\mathcal U_{\gamma, x}$ (\cite[Lem.~4.1]{Suz22b}) and by \eqref{e:RDMF2}, the formula~\eqref{eq:VariousFormsA} readily follows. The Markov property follows by the Markov property of $(\E^{\U(B_r), \QP_r^\eta}, \dom{\E^{\U(B_r), \QP_r^\eta}})$ and \eqref{eq:VariousFormsA}. Thus, all the conditions of Dfn.~\ref{d:core} are satisfied under~\ref{ass:CE} and~\ref{ass:ConditionalClos}.

\item {\it $C^1$-local functions (e.g.\ \cite[Dfn.~II.8]{ChoParYoo98})}. Let $\Omega^*:=\{(x, \gamma) \in \R^n \times \U: x \in \gamma\}$ and we equip $\Omega^*$ with the relative topology of the product topology in $\R^n \times \U$.  Let $C^1_b(\U)$ be defined as the space of bounded $\tau_\mrmv$-continuous functions $u$ satisfying
\begin{enumerate}[(i)]
\item the map $y \mapsto \mathcal U_{\gamma, x}(u)(y)$  is differentiable at $x$ for every $(x, \gamma) \in \Omega^*$. 
\item the map $\Omega^* \ni (x, \gamma) \mapsto \nabla \mathcal U_{\gamma, x}(x)$ is continuous.
\end{enumerate}
A function $u:\U \to \R$ is called {\it local} if $u$ is $\sigma(\pr_{B_r})$-measurable for some~$r>0$, where~$\sigma(\pr_{B_r})$ is the $\sigma$-algebra generated by the map $\pr_{B_r}$. Define
\begin{align} \label{d:LFC1}
C^1_{b, loc}(\U):=\{u \in C^1_b(\U): \ \text{$u$ is local}\comma \limsup_{r \to \infty}\E^{\U, \mu}_r(u)<\infty\} \fstop 
\end{align}
Assume~\ref{ass:Mmu},~\ref{ass:CE} and~\ref{ass:ConditionalClos}, and take~$\mathcal C_r=\mathcal C=C^1_{b, loc}(\U)$ for every $r \in \N$. 
 Then all the conditions of Dfn.~\ref{d:core} are satisfied.
\end{enumerate}
\end{ese}

The following proposition relates the two square fields $\cdc^{\U}_r$ and $\cdc^{\U(B_r)}$. 
\begin{prop}[Truncated form {cf.~\cite[Prop.~4.7]{Suz22b}}]
\label{t:ClosabilitySecond} Assume~\ref{ass:CE} and~\ref{ass:ConditionalClos} and take $\{\mathcal C_r\}_{r \in \N}$ as in Dfn.~\ref{d:core}. 
The following relations hold for~$u \in \mathcal C_r$ for every $r\in \N$: 
\begin{align}\label{eq:p:MarginalWP:0}
& \cdc^{\dUpsilon(B_r)}(u_{r}^\eta)(\gamma) = \cdc^{\dUpsilon}_r(u)(\gamma+\eta_{B_r^c}) \comma \quad \text{$\QP$-a.e.~$\eta$, \ $\mu_{r}^\eta$-a.e.~$\gamma \in \U(B_r)$}  \comma
\\
 &\E^{\U, \mu}_r(u):= \int_\dUpsilon  \E^{\U(B_r), \mu_r^\eta}(u_{r}^\eta) \diff\QP(\eta) = \int_\dUpsilon   \cdc^{\dUpsilon}_r(u) \diff\QP \fstop \notag
\end{align}
As a consequence, the form $(\E^{\U, \QP}_{r}, \mathcal C_r)$ is a densely defined closable Markovian form and 
the closure~$(\E_r^{\U, \QP}, \dom{\E_r^{\U, \QP}})$ is a local Dirichlet form on~$L^2(\mu)$. The $L^2(\QP)$-semigroup and resolvent corresponding to $(\E_{r}^{\U, \mu}, \dom{\E_{r}^{\U, \mu}})$ is denoted by $\sem{T_{r, t}^{\U, \QP}}$ and $\{G_{r, \alpha}^{\dUpsilon, \QP}\}_{\alpha > 0}$ respectively. 

Furthermore, if $\Lip_b(\bar{\mssd}_\U, \QP)\subset \mathcal C_r$, then 
\begin{align}\label{e:RD}
 \cdc^\U_r(u) \le \Lip_{\bar{\mssd}_\U}(u)^2\comma \quad u \in \Lip_b(\bar{\mssd}_\U, \QP) \fstop
\end{align}
\end{prop}

\begin{proof}
Although the idea of the proof is similar to~\cite[Prop.~4.7]{Suz22b}, the core chosen there is different from the core $\mathcal C_r$ here. We therefore give the proof below for the sake of completeness. 

We first prove~\eqref{eq:p:MarginalWP:0}. As the second line of~\eqref{eq:p:MarginalWP:0} is an immediate consequence of the first line and the disintegration formula~\eqref{p:ConditionalIntegration}, we only give the proof of the first line of ~\eqref{eq:p:MarginalWP:0}. Let $u \in \mathcal C_r$. 
Then, the RHS of~\eqref{eq:p:MarginalWP:0} is well-defined on a measurable set~$\Omega$ of $\QP$-full measure by~\ref{i:d:core2} in~Dfn.~\ref{d:core}. 
Let $\Omega_r^\eta$ be the section as defined in~\eqref{e:SEF2}, which is of $\QP_r^\eta$-full measure for $\QP$-a.e.~$\eta \in \Omega$ by~\eqref{p:ConditionalIntegration2}.  
As $\mu_r^\eta$ is absolutely continuous with respect to the Poisson measure~$\pi_{\mssm_r}$ by \ref{ass:CE} and the Poisson measure does not have multiple points almost everywhere, we may assume that every~$\gamma \in \Omega_r^\eta$ does not have multiple points, i.e., $\gamma(\{x\}) \in \{0, 1\}$ for every $x \in B_r$. 
Let $\gamma \in \Omega_r^\eta \cap \U^k(B_r)$. Then, 
\begin{align*}
\cdc^\dUpsilon_r(u) (\gamma+\eta_{B_r^\complement})
&= \sum_{x\in\gamma} \Bigl|\nabla \Bigl( u\tparen{\car_{X\setminus \set{x}} \cdot(\gamma+\eta_{B_r^c})+\delta_\bullet}-u\tparen{\car_{X\setminus\set{x}}\cdot(\gamma+\eta_{B_r^c})} \Bigr)\Bigl|^2(x) 
\\
&=\sum_{x\in\gamma}  \Bigl|\nabla \Bigl( u_{r}^{\eta}\tparen{\car_{X\setminus \set{x}} \cdot\gamma+\delta_\bullet}-u^\eta_{r}\tparen{\car_{X\setminus\set{x}}\cdot\gamma} \Bigr)\Bigr|^2(x) 
\\
&=  \sum_{x\in\gamma} \bigl|\nabla u_{r}^{\eta}\tparen{\car_{X\setminus \set{x}} \cdot\gamma+\delta_\bullet}\bigr|^2(x)
\\
&=  \Bigl|\nabla^{\odot k} \bigl(u_r^\eta \bigr)\Bigr|^2(\gamma)
\\
&=  \cdc^{\dUpsilon(B_r)}(u_{r}^\eta)(\gamma)
\end{align*}
where the first equality is the definition of the square field~$\cdc^\dUpsilon_r$;  the third equality holds as $u^\eta_{r}\tparen{\car_{X\setminus\set{x}}\cdot\gamma}$ does not depend on the variable denoted as~$\bullet$ on which the weak gradient~$\nabla$ operates; the fourth equality followed from the definition of the symmetric gradient operator $\nabla^{\odot k}$, for which we used the fact that $\gamma \in \Omega_r^\eta$ does not have multiple points. 
As this argument holds for arbitrary $k \in \N_0$, \eqref{eq:p:MarginalWP:0} has been shown. 
The local property follows immediately by~\eqref{eq:p:MarginalWP:0} and the local property of $(\E_r^{\U(B_r), \QP_r^\eta}, \dom{\E_r^{\U(B_r), \QP_r^\eta}})$.
The Markov property of~$(\E_{r}^{\U, \mu}, \mathcal C_r)$ follows by~\ref{i:d:core3}. 

We now show the closability. Noting that $\E^{\U(B_r), \QP_r^\eta}$ is closable for $\mu$-a.e.~$\eta$ by~\ref{ass:ConditionalClos},  the superposition form~$(\bar{\E}^{\U, \QP}_r,\dom{\bar{\E}^{\U, \QP}_r})$, which shall be defined below in Dfn.~\ref{d:SPF}, is closed by \cite[Prop.~V.3.1.1]{BouHir91}. As the two forms~$(\E_{r}^{\U, \mu}, \mathcal C_r)$ and~$(\bar{\E}^{\U, \QP}_r,\dom{\bar{\E}^{\U, \QP}_r})$ coincide on $\mathcal C_r$ by definition and $\mathcal C_r\subset \dom{\bar{\E}^{\U, \QP}_r}$ by construction, the closability of~$(\E_{r}^{\U, \mu}, \mathcal C_r)$ is inherited from the closedness of the superposition form $(\bar{\E}^{\U, \QP}_r,\dom{\bar{\E}^{\U, \QP}_r})$. As $\mathcal C_r \subset L^2(\QP)$ is dense by \ref{i:d:core1}, the form~$(\E_{r}^{\U, \mu}, \mathcal C_r)$ is densely defined. As the Markov property extends to the closure (e.g., \cite[Thm.~3.1.1]{FukOshTak11}),  the form $(\E_{r}^{\U, \mu}, \dom{\E_{r}^{\U, \mu}})$ is Markovian as well.  

We now prove~\eqref{e:RD}.  
By the Rademacher-type property of~$\E^{\U(B_r), \QP_r^{k, \eta}}$, 
we have that
\begin{align} \label{e:RD2}
\cdc^{\U(B_r)}(u) \le \Lip_{\mssd_\U}(u)^2  \comma \quad  u \in \Lip(\U(B_r), \mssd_\U) \quad r>0\fstop
\end{align}
In view of the relation between~$\cdc^\U_r$ and~$\cdc^{\U(B_r)}$ in~\eqref{eq:p:MarginalWP:0} and the Lipschitz contraction~\eqref{e:SEF3} of the operator~$(\cdot)_r^\eta$, we concluded~\eqref{e:RD}. 
 The proof is complete.
\end{proof}

\subsection{Superposition form}
\begin{defs}[Superposition Dirichlet form,~e.g., {\cite[Prop.\ V.3.1.1]{BouHir91}}] \label{d:SPF} Assume~\ref{ass:CE} and~\ref{ass:ConditionalClos}.
\begin{align} \label{eq:SP} 
\mathcal D(\bar{\E}^{\U, \mu}_r) &:= \biggl\{u \in L^2(\QP): \ \int_\dUpsilon  \E^{\U(B_r), \QP_r^\eta}(u_{r}^\eta) \diff\QP(\eta)<\infty  \biggr\} \comma
\\
\bar{\E}^{\U, \QP}_r(u) &:=\int_\dUpsilon  \E^{\U(B_r), \QP_r^\eta}(u_{r}^\eta) \diff\QP(\eta) \fstop \notag
\end{align}
It is known that $(\bar{\E}^{\U, \QP}_r,\dom{\bar{\E}^{\U, \QP}_r})$ is a Dirichlet form on~$L^2(\mu)$ \cite[Prop.\ V.3.1.1]{BouHir91}.  The $L^2(\QP)$-semigroup and the infinitesimal generator corresponding to $(\bar{\E}^{\U, \QP}_r,\dom{\bar{\E}^{\U, \QP}_r})$  are denoted by $\sem{\bar{T}_{r, t}^{\U,\QP}}$ and $(\bar{A}_r^{\U, \QP}, \dom{\bar{A}_r^{\U, \QP}})$ respectively. 
\end{defs}
The resolvent $\reso{\bar{G}_{r, \alpha}^{\U, \QP}}$ and the semigroup $\sem{\bar{T}_{r, t}^{\U, \QP}}$ corresponding to the superposition form $\bar{\E}_r^{\U, \QP}$ can be obtained as the superposition of the resolvent~$\reso{G^{\U(B_r), \QP_r^\eta}_{\alpha}}$ and the semigroup $\sem{T^{\U(B_r), \QP_r^\eta}_{t}}$ associated with the form~$\E^{\U(B_r), \QP_r^\eta}$. The following proposition shows that the semigroup (resp. resolvent) corresponding to the superposition form is identified with the superposition of the semigroup (resp.~resolvent), which has been proved by~\cite{LzDS20} in a general framework. 
\begin{prop}[{\cite[(iii) Prop.~2.13]{LzDS20}}]\label{prop: 1-1}
	Assume~\ref{ass:CE} and~\ref{ass:ConditionalClos}.
The following holds: 
	\begin{align} \label{eq: R-1}
		\bar{G}_{r,\alpha}^{\U, \QP}u(\gamma)  = G^{\U(B_r), \QP_r^\gamma}_{\alpha} u_{r}^{\gamma}(\gamma_{B_r})  \comma \quad \bar{T}_{r,t}^{\U, \QP}u(\gamma)  = T^{\U(B_r), \QP_r^\gamma}_{t} u_{r}^{\gamma}(\gamma_{B_r}) \, ,
	\end{align}
	for $\QP$-a.e.\ $\gamma\in \dUpsilon$, every $t>0$.
\end{prop}
\begin{rem}
The proof of \cite[(iii) Prop.~2.13]{LzDS20} has been given in terms of direct integral. As the measure $\QP_r^\eta$ can be identified to the conditional probability~$\QP(\cdot\ |\ \cdot_{B_r^c}=\eta_{B_r^c})$ by a bi-measure-preserving isomorphism as remarked in~\eqref{r:BMP}, our setting can be identified with a particular case of direct integrals in~\cite{LzDS20}.
\end{rem}

As the former form is constructed as the smallest closed extension of $(\E_r^{\U, \QP}, \mathcal C_r)$, it is clear by definition that 
$$\E_r^{\U, \QP}=\bar{\E}_r^{\U, \QP} \quad \text{on} \quad \mathcal C_r \comma \quad \dom{\E_r^{\U, \QP}} \subset \dom{\bar{\E}_r^{\U, \QP}} \fstop$$ 
\begin{ass}\label{t:S=M} We call~\ref{ass:Dom} if 
\begin{align}\tag*{$(\mathsf{D})_{\ref{t:S=M}}$}\label{ass:Dom} 
(\E_r^{\U, \QP}, \dom{\E_r^{\U, \QP}}) = (\bar{\E}_r^{\U, \QP}, \dom{\bar{\E}_r^{\U, \QP}}) \quad r \in \N\fstop
\end{align}
\end{ass}
\begin{rem}\ 
\begin{enumerate}[(i)]
\item For a suitable choice of $\{C_r\}_{r \in \N}$, Assumption~\ref{ass:Dom} has been verified for a Dirichlet form whose invariant measure is ${\sf sine}_\beta$ for every $\beta>0$, see \cite[Thm.~4.11]{Suz22b};
\item \ref{ass:Dom} will be used only for (ii) in Thm.~\ref{thm: Erg} in this paper.
\end{enumerate} 
\end{rem}

Under Assumption~\ref{t:S=M}, Prop.~\ref{prop: 1-1} provides the superposition formula for the resolvent~$\reso{G_{r, \alpha}^{\U, \QP}}$ and the semigroup $\sem{T_{r, t}^{\U, \QP}}$ in terms of the resolvent~$\reso{G_\alpha^{\U(B_r), \QP_r^\eta}}$ and the semigroup $\sem{T_t^{\U(B_r), \QP_r^\eta}}$ respectively. 
\begin{cor}[Coincidence of semigroups]\label{prop: 1} Assume~\ref{ass:CE}, ~\ref{ass:ConditionalClos} and \ref{ass:Dom}. 
The following three operators coincide: 
	\begin{align} \label{eq: R-1}
		&G_{r, \alpha}^{\U, \QP}u(\gamma)  = \bar{G}_{r, \alpha}^{\U, \QP}u(\gamma) =G^{\U(B_r), \QP_r^\gamma}_{\alpha} u_{r}^\gamma(\gamma_{B_r}),
		\\
		&T_{r, t}^{\U, \QP}u(\gamma)  = \bar{T}_{r, t}^{\U, \QP}u(\gamma) =T^{\U(B_r), \QP_r^\gamma}_{t} u_{r}^\gamma(\gamma_{B_r}) \, ,
	\end{align}
	for $\QP$-a.e.\ $\gamma\in \dUpsilon$, every $t>0$.
\end{cor}

\subsection{Monotone limit form}
The following proposition follows immediately from the definitions of the square field~$\Gamma^{\U}_r$ and the core~$\mathcal C_r$.
\begin{prop}[Monotonicity]\label{p:mono}
Assume~\ref{ass:CE} and~\ref{ass:ConditionalClos}.
The form $(\E^{\U, \QP}_r, \dom{\E^{\U, \QP}_r}$ and the square field $\Gamma^\U_r$ are monotone increasing as $r \uparrow \infty$, viz., 
$$\cdc^\U_r(u) \le \cdc^\U_s(u)\comma \quad \E^{\U, \QP}_r(u) \le \E^{\U, \QP}_s(u) \comma \quad \dom{\E^{\U, \QP}_s} \subset \dom{\E^{\U, \QP}_r} \quad r\le s \fstop$$
\end{prop}
\begin{proof}
As $\mathcal C_r$ is a core of the form~$(\E^{\U, \QP}_r, \dom{\E^{\U, \QP}_r})$ and $\mathcal C_s \subset \mathcal C_r$ by Dfn.~\ref{d:core}, it suffices to check~$\cdc^\U_r(u) \le \cdc^\U_s(u)$ on~$\mathcal C_s$, which is a immediate consequence of the definition~\eqref{eq:d:LiftCdCRep}. The proof is complete.
\end{proof}

\begin{defs}[Monotone limit form]
The form $(\E^{\U, \QP}, \dom{\E^{\U, \QP}})$ is defined as the monotone limit:
\begin{align} \label{eq:Temptation} 
\dom{\E^{\U, \QP}}&:=\{u \in \cap_{r>0} \dom{\E^{\U, \QP}_r}: \E^{\U, \QP}(u) = \lim_{r \to \infty}\E^{\U, \QP}_r(u) <\infty\} \comma
\\
\E^{\U, \QP}(u)&:=\lim_{r \to \infty}\E^{\U, \QP}_r(u) \fstop \notag
\end{align} 
The form $(\E^{\U, \QP}, \dom{\E^{\U, \QP}})$ is a Dirichlet form on~$L^2(\mu)$ as it is the monotone limit of Dirichlet forms (e.g., by \cite[Exercise 3.9]{MaRoe90}).
Note that the limit form does not depend on the choice of the exhaustion $\seq{B_r}_{r \in \N}$.
The square field~$\cdc^\U$ is defined as the monotone limit of~$\cdc^\U_r$ as well:
\begin{align} \label{d:SF}
\cdc^\U(u):=\lim_{r \to \infty}\cdc^\U_r(u) \quad u \in \dom{\E^{\U, \QP}}\fstop
\end{align}
 

%
%

\end{defs}

We now show that the form~$(\E^{\U, \QP}, \dom{\E^{\U, \QP}})$ is a local Dirichlet form on~$L^2(\QP)$ and satisfies the Rademacher-type property  with respect to the $L^2$-transportation-type distance~$\bar{\mssd}_\U$.
\begin{prop}\label{p:DF} Assume~\ref{ass:CE} and~\ref{ass:ConditionalClos}.
The form $(\E^{\U, \QP}, \dom{\E^{\U, \QP}})$ is a
local Dirichlet form on~$L^2(\QP)$.
Furthermore,  if $\Lip_b(\bar{\mssd}_\U, \QP) \subset \mathcal C_r$ for every~$r \in \N$, then $(\E^{\U, \QP}, \dom{\E^{\U, \QP}})$ satisfies Rademacher-type property:
\begin{align}\label{p:Rad1}
\Lip_b(\bar{\mssd}_\U, \QP) \subset \dom{\E^{\U, \QP}}\comma \quad \cdc^\U(u) \le \Lip_{\bar{\mssd}_\U}(u)^2 \fstop
\end{align}
\end{prop}
\proof
The local property of $(\E^{\U, \QP}, \dom{\E^{\U, \QP}})$ follows from \eqref{d:SF}. 
We show the Rademacher-type property. 
Since~$\cdc^\U$ is the limit square field of~$\cdc^{\U}_r$ as in~\eqref{d:SF}, it suffices to show 
\begin{align*} 
 \cdc^\U_r(u) \le \Lip_{\bar{\mssd}_\U}(u)^2  \comma\quad u \in \Lip(\bar{\mssd}_\U, \QP) \quad r>0\comma
\end{align*}
which has been already proven in Prop.~\ref{t:ClosabilitySecond}. We verified~\ref{p:Rad}. 
The proof is complete.
\qed

The $L^2$-resolvent operators and the $L^2$-semigroups corresponding to the form~\eqref{eq:VariousFormsA} and the form~\eqref{eq:Temptation} are denoted respectively by 
$$\bigl\{G_{r, \alpha}^{\dUpsilon, \QP}\bigr\}_{\alpha>0},\ \bigl\{T^{\dUpsilon, \QP}_{r, t}\bigr\}_{t >0} \quad \text{and} \quad \bigl\{G_{\alpha}^{\dUpsilon, \QP}\bigr\}_{\alpha>0}, \ \bigl\{T^{\dUpsilon, \QP}_t\bigr\}_{t >0} \fstop$$ 

\begin{prop}\label{prop: MGS} Assume~\ref{ass:CE} and~\ref{ass:ConditionalClos}. 
The semigroup $\{T^{\U, \QP}_{t}\}_{t \ge 0}$ is the $L^2(\QP)$-strong operator limit of the semigroups $\{T^{\U, \QP}_{r, t}\}_{t \ge 0}$, viz., 
$$\text{{\small $L^2(\mu)$--}}\lim_{r \to \infty} G^{\U, \QP}_{r, \alpha} u= G^{\U, \QP}_\alpha u, \quad \text{{\small $L^2(\mu)$--}}\lim_{r \to \infty} T^{\U, \QP}_{r, t} u= T^{\U, \QP}_t u  \quad u \in  L^2(\QP)\comma \quad t>0 \fstop$$
\end{prop}
\begin{proof}
The statement follows from the monotonicity of~$(\E^{\U, \QP}_r, \dom{\E^{\U, \QP}_r}$ as $r\uparrow\infty$ proven in~Prop.~\ref{p:mono}
and~\cite[S.14, p.373]{ReeSim80}. 
\end{proof}

\subsection{Quasi-regularity}\label{subsec: ES}In this subsection, we discuss a sufficient condition for the quasi-regularity. 
\begin{ass}[Quasi-regularity]\label{a:QR}Let $\mathcal F^{\U, \QP} \subset  \dom{\E^{\U, \QP}}$ be any closed Markovian subspace. We call~\ref{ass:QR} for~$\mathcal F^{\U, \QP}$ if 
\begin{align}\tag*{$(\mathsf{QR})_{\ref{a:QR}}$}\label{ass:QR}
\text{$(\E^{\U, \QP}, \mathcal F^{\U, \QP})$ is quasi-regular in $(\U, \tau_\mrmv)$} \fstop
\end{align} 
\end{ass}

In the following, we introduce another monotone limit form having a (possibly) smaller domain. 
\begin{prop}[smaller domain]\label{p:DF2} Assume~\ref{ass:CE} and~\ref{ass:ConditionalClos}. Let $\seq{\mathcal C_r}_{r \in \N}$ be a sequence of algebras in~Dfn.~\ref{d:core}. Then, the form $(\E^{\U, \QP}, \mathcal C)$ defined as 
\begin{align} \label{eq:Temptation2} 
\mathcal C&:=\{u \in \cap_{r \in \N} \mathcal C_r: \lim_{r \to \infty}\E^{\U, \QP}_r(u) <\infty\} \comma
\\
\E^{\U, \QP}(u)&:=\lim_{r \to \infty}\E^{\U, \QP}_r(u) \comma\notag
\end{align} 
is closable. Let $(\E^{\U, \QP}, \mathcal F^{\U, \QP})$ be the closure~$(\E^{\U, \QP}, \overline{\mathcal C})$. Then,  $(\E^{\U, \QP}, \mathcal F^{\U, \QP})$ is a local Dirichlet form on~$L^2(\QP)$. Furthermore,  if either of the following holds for every $r \in\N$
$$\Lip_b(\bar{\mssd}_\U, \QP) \subset \mathcal C_r\comma \quad \Lip_b({\mssd}_\U, \QP) \subset \mathcal C_r\comma \quad \Lip_b(\bar{\mssd}_\U, \tau_\mrmv) \subset \mathcal C_r\comma \quad \Lip_b({\mssd}_\U, \tau_\mrmv) \subset \mathcal C_r \comma$$
then $(\E^{\U, \QP}, \mathcal F^{\U, \QP})$ satisfies Rademacher-type property 
\begin{align}\tag*{$\mathsf{(Rad_{\bar{\mssd}_\U, \QP})}_{\ref{p:DF2}}$}\label{p:Rad}
\Lip_b(\bar{\mssd}_\U, \QP) \subset \mathcal F^{\U, \QP}\comma \quad \cdc^\U(u) \le \Lip_{\bar{\mssd}_\U}(u)^2 \comma
\end{align}
$($resp.~$\mathsf{(Rad_{{\mssd}_\U, \QP})}$, $\mathsf{(Rad_{\bar{\mssd}_\U, \tau_\mrmv})}$ and $\mathsf{(Rad_{{\mssd}_\U, \tau_\mrmv})}$$)$.
\end{prop}
\begin{proof}
As $\mathcal C \subset \dom{\E^{\U, \QP}}$ by definition, the closability of $\mathcal C$ follows by the closedness of $\dom{\E^{\U, \QP}}$ proven in Prop.~\ref{p:DF}. The local property of $\mathcal F^{\U, \QP}$ is inherited from $\dom{\E^{\U, \QP}}$. The Markov property of~$(\E^{\U, \QP}, \mathcal C)$ follows by the Markov property of $(\E^{\U, \QP}_r, \mathcal C_r)$ and \eqref{eq:Temptation2}.  As the Markov property is inherited to the closure by e.g.,~\cite[Thm.~3.1.1]{FukOshTak11}, we concluded that $\mathcal F^{\U, \QP}$ is Markovian. The rest of the arguments follows by the same proofs as in Prop.~\ref{p:DF}.
\end{proof}

\begin{cor}\label{cor:SDR}
Assume~\ref{ass:CE} and~\ref{ass:ConditionalClos}. 
If $\mathcal C_r=\mathcal C$ $(r \in \N)$ is either one of the following:
\begin{align} \label{e:CCL2}
\mathcal C =  \Lip_b(\bar{\mssd}_\U, \tau_\mrmv)  \quad \text{or} \quad \mathcal C =  \Lip_b({\mssd}_\U, \tau_\mrmv) \comma 
\end{align}
and  $\mathcal F^{\U, \QP}$ is the closure $\overline{\mathcal C}$, then $(\EE{\dUpsilon}{\QP}, \mathcal F^{\U, \QP})$ is a quasi-regular local Dirichlet form. 
\end{cor}
\begin{proof}
First of all, the form is closable on~$\mathcal C$ by Prop.~\ref{p:DF2} and the fact that $\mathcal C$ satisfies Dfn.~\ref{d:core} as seen in (b) Example~\ref{ese:CC}. Furthermore, $(\EE{\dUpsilon}{\QP}, \mathcal F^{\U, \QP})$ is a local Dirichlet form and
the Rademacher-type property $\Lip_b(\bar{\mssd}_\U, \tau_\mrmv)$ (resp.~$\Lip_b({\mssd}_\U, \tau_\mrmv))$ holds by Prop.~\ref{p:DF2}. 
Thus, we conclude the quasi-regularity \ref{ass:QR} by the proof of \cite[Cor.\ 6.3]{LzDSSuz21}.
\end{proof}
\begin{rem}[A different core] \label{r:DC}
Another sufficient condition for~\ref{ass:QR} has been studied in~\cite[Thm.~1]{Osa96} by taking a core $\mathcal C_r=\mathscr D_\infty$ in Dfn.~\ref{d:core}, where $\mathscr D_\infty$ is a space of smooth local functions (see, \cite[(0.3)]{Osa96}) and take the domain to be the closure of $\mathscr D_\infty$. We note that functions in the core $\mathcal C$ in \eqref{e:CCL2} are not necessarily local functions. The domain $\mathcal F^{\U, \QP}$ defined as the closure of $\mathcal C$ in Cor.~\ref{cor:SDR} is therefore not necessarily the same as the domain constructed as the closure of $\mathscr D_\infty$ in~\cite{Osa96}.  
\end{rem}

\section{Tail-triviality, Finiteness of $\bar{\mssd}_\dUpsilon$ and Irreducibility} \label{sec: Irr}
\subsection{Irreducibility and Tail-triviality}
We recall that $\ttonde{\EE{\dUpsilon}{\QP}, \mathcal F^{\U, \QP}}$ is {\it irreducible} if $\EE{\dUpsilon}{\QP}(u)=0$ for $u \in  \mathcal F^{\U, \QP}$ implies that $u$ is constant $\QP$-almost everywhere. The following definition corresponds to the irreducibility of the conditioned form~\eqref{eq: condF}. 
Let $\seq{B_{r}}$ be a comapct convex domain exhaustion in $\R^n$.
 \begin{defs}[Conditional irreducibility]\normalfont\label{ass:ConditionalErgodicity}
We say that {\it the conditional irreducibility} (in short: \ref{ass:ConditionalErg}) holds if,
for every $r \in \N$,  $\QP$-a.e.\ $\eta \in \dUpsilon$ and $k \in \mathcal K_r^\eta$, 
\begin{align*}\tag*{$(\mathsf{CI})_{\ref{ass:ConditionalErgodicity}}$}\label{ass:ConditionalErg}
\text{if $u \in \dom{\EE{\dUpsilon(B_r)}{\QP^\eta_{r}}}$ and $\EE{\dUpsilon(B_r)}{\QP^\eta_{r}}(u)=0$, then $u\mrestr{\dUpsilon^{k}(B_r)}=C_{r}^{\eta, k}$ \ $\QP^{\eta, k}_{r}$-a.e.} \comma 
\end{align*}
where $C_{r}^{\eta, k}$ is a constant depending on $r, \eta$ and $k$.
\end{defs}
\begin{rem}\ 
\begin{enumerate}[$(a)$]
\item In terms of the corresponding diffusion process,  Assumption~\ref{ass:ConditionalErg} can be understood as the ergodicity of the interacting {\it finite} particles in~$\dUpsilon(B_r)$ conditioned to be~$\eta_{B_r^c}$ outside~$B_r$. 
\item Assumption \ref{ass:ConditionalErg} can be verified for a wide class of invariant measures $\QP$ such as Gibbs measures including Ruelle measures,  and determinantal/permanental point processes including $\mathrm{sine}_\beta$, $\mathrm{Airy}_\beta$, $\mathrm{Bessel}_{\alpha, \beta}$, Ginibre, which will be discussed in \S \ref{sec: Exa}.
\end{enumerate}
\end{rem}

\smallskip
The main theorem of this section is the following:
\begin{thm}\label{thm: Erg} Let $\QP$ be a Borel probability measure on~$\dUpsilon$ satisfying~\ref{ass:CE} and~\ref{ass:ConditionalClos}, and $\mathcal F^{\U, \QP} \subset  \dom{\EE{\dUpsilon}{\QP}}$ be any closed Markovian subspace. 
\begin{enumerate}[$(i)$]
\item Suppose \ref{ass:Dom}.  Then 
$$ \text{$\ttonde{\EE{\dUpsilon}{\QP}, \dom{\EE{\dUpsilon}{\QP}}}$ is irreducible} \quad \implies \quad \text{ $\QP$ is tail trivial~\ref{ass:T}} $$
\item Suppose \ref{ass:SCE}, \ref{ass:ConditionalErg}, \ref{ass:QR} of~$\mathcal F^{\U, \QP}$ and \ref{ass:Rig}.
Then
$$ \text{ $\QP$ is tail trivial~\ref{ass:T}} \quad \implies \quad \text{$(\EE{\dUpsilon}{\QP}, \mathcal F^{\U, \QP})$ is irreducible} $$
\end{enumerate}
\end{thm}

In the following subsections, we give the proof of Theorem \ref{thm: Erg}.

\subsection{The proof of (i)}
Recall that $\ttonde{\EE{\dUpsilon}{\QP},\dom{\EE{\dUpsilon}{\QP}}}$ is irreducible if and only if $\{T^{\dUpsilon, \QP}_t\}_{t>0}$-invariant sets are trivial (see, e.g., \cite[Prop.\ 2.3 and Appendix]{AlbKonRoe97}), i.e., every $\Xi\subset \dUpsilon$ satisfying
 $$T^{\dUpsilon, \QP}_t (\1_{\Xi}u) = \1_{\Xi}T^{\dUpsilon, \QP}_t u,\quad  u \in L^2(\QP) $$
 satisfies either $\QP(\Xi)=1$ or $\QP(\Xi)=0$. Therefore,  it suffices to show that every set $\Xi \in \mathscr T(\dUpsilon)$ is $\{T^{\dUpsilon, \QP}_t\}_{t>0}$-invariant. 
By Proposition \ref{prop: MGS}, we obtain $T^{\dUpsilon, \QP}_{r, t}u \to T^{\dUpsilon, \QP}_{t}u$ in $L^2(\mu)$ as $r \uparrow \infty$  for every $u \in L^2(\mu)$. Thus, it suffices to show that
\begin{align} \label{eq: TS}
\text{every tail set $\Xi \in \mathscr T(\dUpsilon)$ is $\{T^{\dUpsilon, \QP}_{r, t}\}_{t>0}$-invariant for every $r>0$} \fstop
\end{align}
 Indeed, if it is true, then 
\begin{align*}
T^{\dUpsilon, \QP}_t\1_{\Xi}u = \text{{\small $L^2(\mu)$ --}}\lim_{r\to \infty} T^{\dUpsilon, \QP}_{r, t}\1_{\Xi}u = \text{{\small $L^2(\mu)$ --}}\lim_{r\to \infty} \1_{\Xi}T^{\dUpsilon, \QP}_{r, t}u = \1_{\Xi}T^{\dUpsilon, \QP}_tu \fstop
\end{align*}

We now show \eqref{eq: TS}. By~\eqref{eq: EXP0}, every set $\Xi \in \mathscr T(\dUpsilon)$ has the following expression:
\begin{align} \label{eq: EXP}
\Xi=\U(B_r) + {\rm pr}_{B_r^c}(\Xi), \quad \text{for {\it every}}\ r \in \N\fstop
\end{align}
By Proposition \ref{prop: 1}, for $u \in L^2(\QP)$, 
$$T^{\dUpsilon, \QP}_{r, t}u(\gamma) = T^{\U(B_r), \QP^\gamma_{r}}_t u_{r}^{\gamma}(\gamma_{B_r}), \quad \text{$\mu$-a.e.\ $\gamma$} \fstop$$
By \eqref{eq: EXP}, the function $(\1_{\Xi})_{r}^{\gamma} \equiv 1$ on $\dUpsilon(B_r)$ if and only if $\gamma \in \Xi$, 
which leads to  
$$T^{\U(B_r), \QP^\gamma_{r}}_t (u\1_{\Xi})_{r}^{\gamma}(\gamma_{B_r}) 
= \1_{\Xi}(\gamma)T^{\U(B_r), \QP^\gamma_{r}}_t u_{r}^{\gamma}(\gamma_{B_r}) \fstop$$
Therefore, for every $r \in \N$
 \begin{align*}
 T^{\dUpsilon, \QP}_{r, t}\1_{\Xi} u(\gamma) &=  T^{\U(B_r), \QP^\gamma_{r}}_t(\1_{\Xi}u)_{r}^{\gamma}(\gamma_{B_r}) = \1_\Xi T^{\U(B_r), \QP^\gamma_{r}}_tu_{r}^{ \gamma}(\gamma_{B_r}) = \1_{\Xi}T^{\dUpsilon, \QP}_{r, t} u(\gamma) \comma
 \end{align*}
for $\mu$-a.e.\ $\gamma$.
The proof of (i)  is complete.
\qed

\subsection{The proof of (ii)}
By the number rigidity~\ref{ass:Rig}, we can take a measurable set $\Omega_{\sf rig}^r \subset \U$ so that $\mu(\Omega_{\sf rig}^{r})=1$ and if~$\gamma, \eta \in \Omega_{\sf rig}^r$ with $\gamma_{B_r^c} = \eta_{B_r^c}$, then $\gamma(B_r) = \eta(B_r)$. Let $\Omega_{\sf rig} = \cap_{r \in \N} \Omega_{\sf rig}^{r}$, which is of $\QP$-full measure as well.  
Let $u \in \mathcal F^{\U, \QP}$ so that $ \cdc^{\dUpsilon}(u) =0$. By the monotonicity~\eqref{d:SF}, we have $\cdc^{\U}_r(u)=0$. 
By the formula~\eqref{eq:p:MarginalWP:0} and the same proof as \cite[Prop.\ 5.14]{LzDSSuz21},  for every $r \in \N$, there exists $\Omega^r_0 \subset \dUpsilon$ so that $\QP(\Omega^r_0)=1$ and 
$$\cdc^{\dUpsilon(B_{r})}(u_{r}^{\eta}) =0 \qquad \text{$\QP_{r}^\eta$-a.e.} \comma \quad \eta \in \Omega^r_0 \fstop$$
By Assumption~\ref{ass:ConditionalErg},  for $r \in \N$, there exists a measurable set $\Omega^r_{\sf rig, ic}\subset \Omega_0^r \cap \Omega_{\sf rig}$ of full $\mu$-measure so that, for every $\eta \in \Omega^r_{\sf rig, ic}$, there exists  $k=k(\eta) \in \N_0$ and a constant $C^{\eta, k}_{r}$ satisfying that 
\begin{align} \label{eq: E: erg0}
u_{r}^{\eta}\equiv C^{\eta, k}_{r} \qquad \QP_{r}^\eta\text{-a.e.} \fstop
\end{align} 
Note that the measure $\QP_{r}^\eta$ is fully supported on $\dUpsilon^{k(\eta)}(B_{r})$ by the number rigidity~\ref{ass:Rig} and~\ref{ass:SCE}.
Let $\Omega_{\sf rig, ic}:= \cap_{r \in \N}\Omega^r_{\sf rig, ic}$, which satisfies $\mu(\Omega_{\sf rig, ic})=1$. 

By the quasi-regularity~\ref{ass:QR}, there exists a quasi-continuous $\QP$-version $\tilde{u}$ of $u$ (see \cite[Prop.\ 3.3 in Chap.~IV]{MaRoe90}). 
Therefore, we can take a closed nest so that $\tilde{u}$ is $\T_\mrmv$-continuous on~$K_m$ for every~$m \in \N$.
Define~$\Omega_{\sf qc}:=\cup_{m \in \N} K_m$, which is of $\mu$-full measure since $\Omega_{\sf qc}^c$ is an exceptional set with respect to~$\ttonde{\EE{\dUpsilon}{\QP}, \mathcal F^{\U, \QP}}$.
Up to relabelling $K_m$, we may therefore assume that $\QP(K_m) >1-\frac{1}{2m}$.
Let $\Omega_m:=\Omega_{{\sf rig, ic}} \cap K_m$ for $m \in \N$. 
Since $\tilde{u}$ is $\T_\mrmv$-continuous on $\Omega_m$, the function $\tilde{u}_{r}^{\eta}$ is continuous on $(\Omega_{m})_r^\eta$ for every~$\eta \in \Omega_m$ and $r \in \N$ where $(\Omega_m)_{r}^\eta$ is the section defined in~\eqref{e:SEF2}. 
By Prop.~\ref{p:ConditionalIntegration}, we have that
\begin{align} \label{eq: CEIR}
\QP(\Omega_m)= \int_{\dUpsilon}\QP_{r}^\eta \bigl((\Omega_{m})_r^\eta\bigr) \diff \QP(\eta) \fstop 
\end{align}
Thus, by noting that $\QP(\Omega_m) > 1-\frac{1}{2m}$ and the integrand of the r.h.s.\ of \eqref{eq: CEIR}  is non-negative and bounded from above by $1$, there exists $\Omega_{m, r} \subset \dUpsilon$ so that $\QP(\Omega_{m, r}) > 1-\frac{1}{2m}$ and 
\begin{align} \label{eq: Ir: 11}
\QP_{r}^\eta \bigl((\Omega_{m})_r^\eta\bigr)>0 \qquad \forall \eta \in  \Omega_{m, r}\fstop
\end{align}
Define $\Omega_{m}^r:= \Omega_m \cap \Omega_{m, r}$. As $\QP(\Omega_{m, r}), \QP(\Omega_m) > 1-\frac{1}{2m}$, by Inclusion-Exclusion formula, it holds that
\begin{align} \label{eq: Ir 111}
\QP(\Omega_{m}^r) > 1-\frac{1}{m}\comma  \quad m \in \N\fstop
\end{align}
Combining~\eqref{eq: E: erg0}~and~\eqref{eq: Ir: 11} with the fact that  $\QP_{r}^\eta\mrestr{{\dUpsilon^{k}(B_r)}}$ is fully supported in $\dUpsilon^{k}(B_r)$ and $\tilde{u}$ is $\T_\mrmv$-continuous on~$(\Omega_{m})_r^\eta$,  we obtain that 
\begin{align} \label{eq: E: erg1}
\text{$\tilde{u}_{r, \eta}\equiv C^k_{r, \eta}$ {\it everywhere} in $(\Omega_{m})_r^\eta$ for every $\eta \in \Omega_{m}^r$}\, \fstop
\end{align}  
By Lemma \ref{lem: IEF} in Appendix applied to $\Omega_m^r$ in \eqref{eq: Ir 111},  we can take $n\mapsto m_n \in \N$ with $m_{n} \le m_{n'}$ for $n \le n'$ so that, by taking 
$\Omega=\limsup_{n \to \infty}\cap_{r=1}^n\Omega_{m_n}^r$,
it holds that 
$$\QP(\Omega) = 1 \fstop$$ 

We now prove that 
\begin{align} \label{eq: m: 2}
\text{$\tilde{u}$ is constant~$\QP\text{-a.e.}$~on $\Omega$} \fstop
\end{align}
\paragraph{Claim 1}The statement~\eqref{eq: m: 2} holds if the following statement is true: for every $\Xi_1, \Xi_2 \subset \Omega$ with $\mu(\Xi_1)\mu(\Xi_2)>0$, there exist $\gamma^1 \in \Xi_1$ and $\gamma^2 \in \Xi_2$ so that 
 \begin{align} \label{eq: chyp}
 \tilde{u}(\gamma^1) =\tilde{u}(\gamma^2)\fstop
 \end{align}
 \proof[Proof of Claim 1]
Assume that the statement~\eqref{eq: chyp} is true. Take $\Xi_1=\{\tilde{u}>a\}$ and $\Xi_2=\{\tilde{u} \le a\}$ for $a \in \R$. If there exists $a \in \R$ so that $\mu(\Xi_1)\mu(\Xi_2)>0$, then this contradicts \eqref{eq: chyp}. Thus, there is no such $a \in \R$, which means  
$\QP(\Xi_1)\QP(\Xi_2)=0$ for every $a \in \R$.  This concludes that $\tilde{u}$ is constant $\QP$-a.e.\ on $\Omega$.
 \qed
\smallskip

We thus only have to prove \eqref{eq: chyp}. Without loss of generality, we may assume $\QP(\Xi_1)>0$. Since $\mu$ is tail trivial, $\mu(\Xi_1)>0$ and $\Xi_1 \subset \mathcal T(\Xi_1)$, it holds that $\mu(\mathcal T(\Xi_1))=1$, where $\mathcal T(\Xi_1)$ is the tail set of $\Xi_1$ as defined in~\eqref{eq: ts}. Thus, $\mu(\mathcal T(\Xi_1) \cap \Xi_2)>0$, and $\mathcal T(\Xi_1) \cap \Xi_2$ is non-empty. Take an element $\gamma^2 \in \mathcal T(\Xi_1) \cap \Xi_2$. By the definition \eqref{eq: ts} of the tail set $\mathcal T(\Xi_1)$ and $\Omega \subset \Omega_{\sf rig}$, there exists $r_0 \in \N$ and $\gamma^1 \in \Xi_1$  so that
 \begin{align} \label{eq: LSP-2}
 \gamma^1_{B_{r_0}^c} = \gamma^2_{B_{r_0}^c}, \quad \gamma^1({B_{r_0}})=\gamma^2({B_{r_0}}) \fstop
 \end{align}
\paragraph{Claim 2}$\gamma^1_{B_{j}}, \gamma^2_{B_{j}}\in (\Omega_{m_j})_j^{\gamma^2}$ for some $j \in \N$.
\proof[Proof of Claim 2]
Recall
\begin{align} \label{eq: LSP}
\Omega=\limsup_{n \to \infty}\cap_{r=1}^n\Omega_{m_n}^r := \bigcap_{n \ge1} \bigcup_{j \ge n} \bigcap_{r=1}^j\Omega_{m_j}^r \fstop
\end{align}
As $\gamma^1, \gamma^2 \in \Omega$, there exist $j_1, j_2 \in \N$ with $j_1, j_2 \ge r_0$ so that 
$$ \gamma^1 \in \bigcap_{r=1}^{j_1}\Omega_{m_{j_1}}^r \comma \quad \gamma^2 \in \bigcap_{r=1}^{j_2}\Omega_{m_{j_2}}^r \comma\quad \text{in particular}\comma \quad  \gamma^1  \in \Omega_{m_{j_1}}^{j_1} \quad \gamma^2 \in \Omega_{m_{j_2}}^{j_2} \fstop$$
We may assume without loss of generality $j_1 \le j_2$. As $\Omega_{m_{j_1}}^{j_1}\subset \Omega_{m_{j_1}}$ and $\Omega_{m_{j_2}}^{j_2}\subset \Omega_{m_{j_2}}$ by definition, and the monotonicity~$\Omega_{m_{j_1}} \subset  \Omega_{m_{j_2}}$ by construction, we have
\begin{align}\label{LSP-22}
\gamma^1 \in \Omega_{m_{j_1}}^{j_1} \subset \Omega_{m_{j_1}} \subset  \Omega_{m_{j_2}} \comma \quad \gamma^2 \in \Omega_{m_{j_2}}^{j_2}\subset\Omega_{m_{j_2}} \fstop
\end{align}
As $j_2 \ge r_0$, \eqref{eq: LSP-2} implies
\begin{align} \label{eq: LSP-222}
 \gamma^1_{B_{j_2}^c} = \gamma^2_{B_{j_2}^c}, \quad \gamma^1({B_{j_2}})=\gamma^2({B_{j_2}})=:k \fstop
 \end{align}
By~\eqref{LSP-22} and~\eqref{eq: LSP-222}, we obtain
\begin{align} \label{eq: LSP-2222}
\gamma^1_{B_{j_2}}, \gamma^2_{B_{j_2}}\in (\Omega_{m_{j_2}})_{j_2}^{\gamma^2} \fstop
\end{align}
\qed

We now resume the proof of \eqref{eq: chyp}. In view of~\eqref{eq: E: erg1} and~\eqref{eq: LSP-2222}, we conclude 
\begin{align*}
\tilde{u}(\gamma^1) &=  \tilde{u}(\gamma^1_{B_{j_2}}+\gamma^1_{B_{j_2}^c}) = \tilde{u}_{j_2}^{\gamma^2}(\gamma^1_{B_{j_2}})
= C^{\gamma^2, k}_{j_2} 
=\tilde{u}_{j_2}^{\gamma^2}(\gamma^2_{B_{j_2}})
=\tilde{u}(\gamma^2_{B_{j_2}}+\gamma^2_{B_{j_2}^c}) 
=\tilde{u}(\gamma^2) \comma
\end{align*}
which proves \eqref{eq: chyp}. The proof is completed.
\qed

\bigskip
For~a closed Markovian subspace~$\mathcal F^{\U, \QP} \subset  \dom{\EE{\dUpsilon}{\QP}}$, let $S_t^{\U, \QP}$ be the corresponding $L^2(\QP)$-semigroup and $(L^{\U, \QP}, \dom{L^{\dUpsilon, \QP}})$ be the infinitesimal generator respectively.
\begin{cor}\label{cor: Erg0}
Let $\QP$ be a Borel probability measure on~$\U$ and $\mathcal F^{\U, \QP} \subset  \dom{\EE{\dUpsilon}{\QP}}$ be any closed Markovian subspace.  Suppose~\ref{ass:SCE}, \ref{ass:ConditionalClos}, \ref{ass:ConditionalErg},~\ref{ass:QR} of~$\mathcal F^{\U, \QP}$, ~\ref{ass:Rig} and~\ref{ass:T}. Then, the following hold:
\begin{enumerate}[$(i)$]
\item  $\ttonde{\EE{\dUpsilon}{\QP}, \mathcal F^{\U, \QP}}$ is irreducible;
\item $\{S^{\dUpsilon, \QP}_t\}$ is irreducible, i.e., every $\Xi\in \mathscr B(\tau_\mrmv)^\QP$ with
 $$S^{\dUpsilon, \QP}_t (\1_{\Xi}f) = \1_{\Xi}S^{\dUpsilon, \QP}_t f,\quad  f \in L^2(\QP) $$
 satisfies either $\QP(\Xi)=1$ or $\QP(\Xi)=0$;
\item $\{S^{\dUpsilon, \QP}_t\}$ is ergodic, i.e., 
\begin{align*}
\int_{\dUpsilon} \biggl( S^{\dUpsilon, \QP}_t u - \int_{\dUpsilon} u \diff \QP \biggr)^2 \diff \QP \xrightarrow{t \to \infty} 0, \quad u \in L^2(\QP);
\end{align*}
\item $L^{\dUpsilon, \QP}$-harmonic functions are trivial, i.e., 
$$\text{If}\ u \in \dom{L^{\dUpsilon, \QP}}\ \text{and}\  L^{\dUpsilon, \QP}u =0, \quad \text{then}\  \ u=\text{const.}\,.$$
\end{enumerate}
\end{cor}
\proof
The statement~(i) is the consequence of Theorem~\ref{thm: Erg}. The equivalences (i)$\iff$(ii)$\iff$(iii)$\iff$(iv) are standard in functional analysis. We refer the readers to, e.g., \cite[Prop.\ 2.3 and Appendix]{AlbKonRoe97}. 
\qed
\smallskip

Let $(\mathbf X_t, \mathbf P_\gamma)$ be a Markov process associated with the quasi-regular Dirichlet form $\ttonde{\EE{\dUpsilon}{\QP},\mathcal F^{\U, \QP}}$ (see \cite[Thm.\ 3.5 in Chap.\ IV]{MaRoe90}). We write $\mathbf P_{\nu}$ for $\int_{\dUpsilon} \mathbf P_\gamma(\cdot) d\nu(\gamma)$ for a bounded Borel measure $\nu$ on $\dUpsilon$. Recall that $\mathcal F^{\U, \QP}_e$ is the extended domain of $\mathcal F^{\U, \QP}$ defined in \eqref{d:EDD}. 
\begin{cor}\label{cor: Erg1}
Let~$\QP$ be a Borel probability measure on~$\U$~and $\mathcal F^{\U, \QP} \subset  \dom{\EE{\dUpsilon}{\QP}}$ be any closed Markovian subspace.  Suppose~\ref{ass:SCE}, \ref{ass:ConditionalClos}, \ref{ass:ConditionalErg},~\ref{ass:QR} of~$\mathcal F^{\U, \QP}$, \ref{ass:Rig}, \ref{ass:T} and $\1 \in \mathcal F^{\U, \QP}_e$. 
Then, the following hold: 
\begin{enumerate}[$(i)$]
\item  for every Borel measurable $\QP$-integrable function $u$, it holds $\mathbf P_\QP$-a.s.\ that
\begin{align} \label{eq: TE3}
\lim_{t \to \infty} \frac{1}{t} \int_0^t u(\mathbf X_s)ds = \int_{\dUpsilon} u \diff \QP;
\end{align}
\item for every non-negative bounded function $h$, \eqref{eq: TE3} holds in $L^1(\mathbf P_{h\cdot\QP})$; 
\item the convergence \eqref{eq: TE3} holds $\mathbf P_\gamma$-a.s.\ for $\EE{\dUpsilon}{\QP}$-q.e.\ $\gamma$.
\end{enumerate}
\end{cor}
\proof
The form $\ttonde{\EE{\dUpsilon}{\QP},\mathcal F^{\U, \QP}}$ is irreducible by Theorem~\ref{thm: Erg}.
Furthermore, it is recurrent as $\1 \in \mathcal F^{\U, \QP}_e$ and $\EE{\dUpsilon}{\QP}(\mathbf 1)=0$, see \cite[Thm.~1.6.3]{FukOshTak11}. 
Therefore, by \cite[Thm.\ 4.7.3]{FukOshTak11}, the proof is complete (although \cite[Thm.\ 4.7.3]{FukOshTak11} assumes the local compactness of the state space, the same proof applies verbatim). 
\qed

\subsection{Finiteness of $\bar{\mssd}_\dUpsilon$}
Recall that $\bar{\mssd}_{\dUpsilon}^\QP(\Xi, \Lambda)$  has been defined in \eqref{d:DFS}  for $\Xi, \Lambda \subset \U$.
 In this subsection, we investigate relations among the tail triviality~\ref{ass:T}, the irreducibility, and the finiteness $\bar{\mssd}_{\dUpsilon}^\QP(\Xi, \Lambda)<\infty$. Namely, we discuss  relations among the following statements:
\begin{enumerate}[$(a)$]
\item $\mu$ is tail trivial~\ref{ass:T};
\item $\bar{\mssd}_{\dUpsilon}^\QP(\Xi, \Lambda)<\infty$ whenever $\Xi \in \mathscr B(\tau_\mrmv)^\QP, \Lambda \in \mathscr B(\tau_\mrmv)$ and $\mu(\Xi), \mu(\Lambda)>0$;
\item $\ttonde{\EE{\dUpsilon}{\QP},\mathcal F^{\U, \QP}}$ is irreducible.
\end{enumerate}

\begin{thm} \label{thm: Equiv} 
Let~$\QP$ be a Borel probability measure on~$\U$. 
Then,  
\begin{itemize}
\item \ref{c:FD}  $\implies$ \ref{c:TT};
\item if \ref{ass:Rig} holds,  then \ref{c:TT} $\implies$ \ref{c:FD}.
\end{itemize}
Suppose that $\QP$ satisfies~\ref{ass:SCE} and~\ref{ass:ConditionalClos},  and $\mathcal F^{\U, \QP} \subset  \dom{\EE{\dUpsilon}{\QP}}$ is any closed Markovian subspace. Then the following hold.
\begin{itemize}
\item if  \ref{ass:ConditionalErg}, \ref{ass:QR}  and \ref{ass:Rig} hold,  then \ref{c:FD} $\implies$ \ref{c:IR};
\item if \ref{p:Rad} holds,  then \ref{c:IR} $\implies$ \ref{c:FD}.
\end{itemize}

\end{thm}
\begin{proof}{\it \ref{c:FD}  $\implies$ \ref{c:TT}}. We prove by contradiction. Assume \ref{c:FD} but $\QP$ is not tail trivial. Then, there exists a tail-measurable set $\Xi \in \mathscr T(\U)$ so that $\QP(\Xi), \QP(\Xi^c)>0$. Note that  $\Xi \in \mathscr B(\tau_\mrmv)$ as  $\mathscr T(\U) \subset \mathscr B(\tau_\mrmv)$ by construction. Note also that $\mathcal T(\Xi)=\Xi$ and $\mathcal T(\Xi^c)=\Xi^c$ as $\Xi, \Xi^c \in \mathscr T(\U)$, where $\mathcal T(\Xi)$ is the tail set of $\Xi$ defined in \eqref{eq: ts}.
By \ref{c:FD}, we have 
\begin{align}
\bar{\mssd}_{\dUpsilon}^\QP(\Xi, \Xi^c)<\infty \fstop
\end{align}
By~\eqref{e:LLR2}, this implies that there exists $\gamma^1 \in \Xi$, $\gamma^2 \in \Xi^c$ and $r \in \N$ so that 
$$\gamma^1_{B_r^c} = \gamma^2_{B_r^c} \comma \quad \gamma^1(B_r)=\gamma^2(B_r) \fstop$$
This however means that $\gamma^1, \gamma^2 \in \mathcal T(\Xi) \cap \mathcal T(\Xi^c) = \Xi \cap \Xi^c=\emptyset$, which is a contradiction. 
\smallskip

{\it \ref{c:TT} $\implies$ \ref{c:FD}}. We prove by contradiction. Assume \ref{c:TT} but \ref{c:FD} does not hold. Then, there exist $\Xi, \Lambda \subset \U$ with $\QP(\Xi), \QP(\Lambda)>0$ so that $\bar{\mssd}_{\dUpsilon}^\QP(\Xi, \Lambda)=\infty$.  By modifying a $\QP$-negligible set, we may assume without loss of generality that 
\begin{align}\label{e:DZP}
\bar{\mssd}_{\dUpsilon}(\cdot, \Lambda) = \infty \quad \text{everywhere on $\Xi$} \fstop
\end{align}
Let $\Omega_{\sf rig}$ be the set defined in the proof of (ii) of Thm.~\ref{thm: Erg}.   Let $\tilde{\Lambda}:=\Lambda \cap \Omega_{\sf rig}$. As $\tilde{\Lambda} \subset \Lambda$, we have
\begin{align}\label{ineq:DZPP}
\infty=\bar{\mssd}_\U(\cdot, \Lambda) \le \bar{\mssd}_\U(\cdot, \tilde{\Lambda}) \quad \text{everywhere on $\Xi$}\fstop
\end{align}
By \ref{c:TT}, we have $\QP(\mathcal T(\Xi))=\QP(\mathcal T(\tilde{\Lambda}))=1$ as $\QP(\Xi), \QP(\tilde{\Lambda})>0$, $\Xi \subset \mathcal T(\Xi)$ and $\tilde{\Lambda}\subset \mathcal T(\tilde{\Lambda})$. Therefore, $\QP(\mathcal T(\Xi) \cap \tilde{\Lambda})>0$ and $\mathcal T(\Xi) \cap \tilde{\Lambda} \neq \emptyset$. Take $\gamma \in \mathcal T(\Xi) \cap \tilde{\Lambda}$. By the defnition of the tail-operation $\mathcal T$ and~\ref{ass:Rig}, there exists $\eta \in \Xi$ and $r \in \N$ so that 
$$\gamma_{B_r^c}= \eta_{B_r^c} \comma \quad \gamma(B_r) = \eta(B_r) \fstop$$
Thus, by~\eqref{e:LLR2}, we obtain $\bar{\mssd}_\U(\gamma, \eta)<\infty$, which contradicts \eqref{ineq:DZPP}. 
\smallskip

{\it \ref{c:FD} $\implies$ \ref{c:IR}}. By \ref{c:FD} $\implies$ \ref{c:TT} and \ref{c:TT} $\implies$ \ref{c:IR} by (ii) of Thm.~\ref{thm: Erg}, we conclude~\ref{c:IR}. 

 \smallskip
 
{\it \ref{c:IR} $\implies$ \ref{c:FD}}. By~\ref{p:Rad}, Prop.~\ref{p:UMD} and $\bar{\mssd}_\dUpsilon(\cdot, \Lambda) \in \Lip_b^1(\bar{\mssd}, \QP)$, it holds that
$$\bar{\mssd}_\dUpsilon(\cdot, \Lambda)\wedge c \in \mathcal F^{\U, \QP}, \quad \cdc^{\dUpsilon, \QP}(\bar{\mssd}_\dUpsilon(\cdot, \Lambda) \wedge c) \le 1 \comma \quad c >0\fstop $$
Let $\bar{\sf d}_{\mu, \Lambda}$ be the maximal function associated with $\ttonde{\EE{\dUpsilon}{\QP},\mathcal F^{\U, \QP}}$ defined in~\eqref{d:HRMF}. 
By the definition of the maximal function $\bar{\sf d}_{\mu, \Lambda}$, we obtain that 
\begin{align*}
\bar{\mssd}_\dUpsilon(\cdot, \Lambda) \wedge c \leq \bar{\sf d}_{\mu, \Lambda} \wedge c \quad \text{$\QP$-a.e.}\fstop
\end{align*}
Passing to the limit $c \to \infty$, we obtain
\begin{align}\label{e:IST1}
\bar{\mssd}_{\dUpsilon}(\cdot, {\Lambda}) \le  \bar{\sf d}_{\mu, {\Lambda}} \comma \quad \Lambda \in \mathscr B(\tau_\mrmv) \fstop
\end{align}
which leads to
\begin{align*}
\bar{\mssd}^\QP_{\dUpsilon}(\Xi,\Lambda) \leq  \mu\text{-}\essinf_{\Xi}\bar{\sf d}_{\mu, \Lambda}\fstop
\end{align*}
By~\cite[Lem.\ 2.16]{HinRam03}, \ref{c:IR} implies 
$\mu\text{-}\essinf_{\Xi}\bar{\sf d}_{\mu, \Lambda}<\infty,$
which concludes \ref{c:FD}. 
\end{proof}
\begin{cor} \label{cor: Equiv} 
Let~$\QP$ be a Borel probability measure on~$\U$ satisfying \ref{ass:SCE}, ~\ref{ass:ConditionalClos}, and let $\mathcal F^{\U, \QP} \subset  \dom{\EE{\dUpsilon}{\QP}}$ be any closed Markovian subspace. Then the following hold.
\begin{itemize}
\item If  \ref{ass:ConditionalErg}, \ref{ass:QR}  and \ref{ass:Rig}  hold, then 
$$ \text{$\QP$ is tail trivial} \quad \implies \quad \text{$(\EE{\dUpsilon}{\QP}, \mathcal F^{\U, \QP})$ is irreducible} \;$$
\item If \ref{p:Rad} holds for $\mathcal F^{\U, \QP}$,
$$\quad \text{$(\EE{\dUpsilon}{\QP}, \mathcal F^{\U, \QP})$ is irreducible}\quad \implies   \quad  \text{$\QP$ is tail trivial} \fstop$$
\end{itemize}
\end{cor}

\begin{rem}
We proved the implication  \ref{c:IR} $\implies$ \ref{c:FD} in Thm.~\ref{thm: Erg} under \ref{ass:Dom} with the domain $\dom{\EE{\dUpsilon}{\QP}}$. The same implication was proved in Cor.~\ref{cor: Equiv} under a different assumption~\ref{p:Rad} with a smaller domain $\mathcal F^{\U, \QP}$. The assumption \ref{ass:Dom} is a condition for the truncated forms~$\E_r^{\U, \QP}$ while \ref{p:Rad} is a condition for $(\E^{\U,\QP}, \F^{\U, \QP})$. We do not have a simple comparison of these two different conditions: as the irreducibility with a smaller domain is a weaker statement than that with a larger domain, Cor.~\ref{cor: Equiv} looks providing the tail-triviality under a weaker assumption than Thm.~\ref{thm: Erg}. However, we do not know whether \ref{p:Rad} is weaker than \ref{ass:Dom}.  For the verification, Cor.~\ref{cor: Equiv} is more convenient  as will be seen in Section~\ref{sec: Ver}.
\end{rem}

\section{Verifications of the main assumptions} \label{sec: Ver}
In this section, we provide sufficient conditions for the verification of the main assumptions in Theorems~\ref{thm: Erg},~\ref{thm: Equiv}.  
See~Examples~\ref{exa: TT},~\ref{exa: R} for the tail triviality~\ref{ass:T} and the number rigidity~\ref{ass:Rig}, and see \S\ref{subsec: ES} for the quasi-regularity~\ref{ass:QR}. 

\paragraph{Quasi-Gibbs measures}We recall the definition of \emph{quasi-Gibbs measures}.
Several slightly different (possibly \emph{non}-equivalent) definitions for this concept were introduced by H.~Osada, see e.g.~\cite[Dfn.~2.1]{Osa13},~\cite[Dfn.\ 3.1]{OsaTan14}, ~\cite[Dfn.~5.1]{Osa19}, or~\cite[Dfn.~2.2]{OsaTan20}.

Let~$\Phi\colon \R^n\rar\R$ be $\mathscr B(\tau_\mrmv)$-measurable, and by~$\Psi\colon (\R^n)^\tym{2}\rar\R$ be $\mathscr B(\tau_\mrmv)^\otym{2}$-measurable and symmetric.
The function~$\Phi$ will be called the \emph{free potential}, and~$\Psi$ the \emph{interaction potential}.
These potentials define a \emph{Hamiltonian}~$\msH_r\colon \dUpsilon\rar \R$ as
\begin{align*}
\msH_r\colon \gamma\longmapsto \Phi^\trid \gamma_{B_r}+\tfrac{1}{2}\Psi^\trid \ttonde{\gamma_{B_r}^\otym{2}}\comma\qquad \gamma\in \dUpsilon\fstop
\end{align*}
Recall that $\mathcal K_r^\eta:=\{k \in \N_0: \QP_r^\eta(\U^k(B_r))>0\}$ has been defined in Dfn.~\ref{d:ConditionalAC}.
\begin{defs}[Quasi-Gibbs measures, {cf.~\cite[Dfn.~2.2]{OsaTan20}}]\label{d:QuasiGibbs}
We say that a Borel probability $\QP$ on~$\U$ is a \emph{$(\Phi,\Psi)$-quasi-Gibbs measure} if there exists a sequence $\{B_r\}_{r\in \N}$ of compact monotone increasing domains covering $\R^n$ so that, for $\QP$-a.e.~$\eta\in\dUpsilon$, every~$r \in \N$, every~$k\in \mathcal K_r^\eta$, there exists a constant~$c_{r,\eta,k}>0$ so that
\begin{align}\label{eq:localACquasiGibbs}
c_{r,\eta,k}^{-1}\, e^{-\msH_r} \cdot \PP_{\mssm_r}\mrestr{\dUpsilon^{k}(B_r)} \ \leq \QP_r^{\eta, k} \ \leq \ c_{r,\eta,k}\, e^{-\msH_r} \cdot \PP_{\mssm_r}\mrestr{\dUpsilon^{k}(B_r)}\fstop
\end{align}
For quasi-Gibbs measures, \ref{ass:SCE} follows immediately by \eqref{eq:localACquasiGibbs}. 
\end{defs}

\begin{rem}\label{r:QuasiGibbs} \ 
\begin{enumerate}[$(a)$]
\item\label{i:r:QuasiGibbs:1} The definition of quasi-Gibbs measures in~\cite[Dfn.~2.2]{OsaTan20} looks slightly different from Dfn.~\ref{d:QuasiGibbs} as we assume~\eqref{eq:localACquasiGibbs} only for~$k\in \mathcal K_r^\eta$ in place of $k\in \N$. These two definitions are, however, equivalent since the definitions of~$\QP_r^{\eta, k}$ in this article is {\it the restriction} on $\U^k(B_r)$:
$$\QP_r^{\eta, {k}}:=\QP_r^\eta\mrestr{\U^k(B_r)} \comma$$ 
 while the corresponding measure in~\cite[Dfn.~2.2]{OsaTan20} has been defined as the measure {\it conditioned} on $\U^k(B_r)$.
\item\label{i:r:QuasiGibbs:3}``$\QP$ belongs to $(\Phi,\Psi)$-quasi-Gibbs measures'' {\it does not necessarily} mean that $\QP$ is governed by the free potential $\Phi$ in the sense of the DLR equation. The symbol $\Phi$ here just plays a role as {\it representative} of the class of $(\Phi,\Psi)$-quasi-Gibbs measures {\it modulo perturbations by adding locally finite free potentials}. To be more precise, 
noting that the constant $c_{r,\eta,k}$ can depend on $r, \eta, k$, if $\QP$ is $(\Phi,\Psi)$-quasi-Gibbs, then $\QP$ is $(\Phi+\Phi',\Psi)$-quasi-Gibbs as well whenever $\Phi'|_{B_r}$ is bounded for every $r \in \N$. Therefore, in this case, we may write $(0,\Psi)$-quasi-Gibbs instead of $(\Phi,\Psi)$-quasi-Gibbs. 
\end{enumerate}
\end{rem}

\begin{ese}[See \cite{Osa19}]\label{r:QuasiGibbsEx}
The class of quasi-Gibbs measures includes all canonical Gibbs measure, and the laws of some determinantal/permnental point processes, as for instance:
\begin{enumerate}
\item\label{i:r:QuasiGibbsEx:1} mixed Poisson measures;
\item canonical Gibbs measures;
\item\label{i:r:QuasiGibbsEx:3} the laws of some determinantal/permanental point processes and related point processes, e.g.,  $\mathrm{sine}_\beta$,~$\mathrm{Bessel}_{\alpha,\beta}$, $\mathrm{Airy}_{\beta}$ ($\beta=1,2,4$) and Ginibre point processes.
\end{enumerate}
\end{ese}

\subsection{Assumption~\ref{ass:ConditionalClos}}
According to~Remark~\ref{r:HMZ} and~\eqref{eq:localACquasiGibbs}, Conditional Closability~\ref{ass:ConditionalClos} holds if 
\begin{align} \label{a:CCC}
e^{-\msH_r}|_{\U^k(B_r)} \in C_b\bigl(\U^k(B_r)\bigr) \quad k \in \N_0\, \quad r \in \N\fstop
\end{align}

\begin{rem}While \eqref{a:CCC} is sufficient to cover all the examples discussed in Section~\ref{sec: Exa}, it is not necessary for~\ref{ass:ConditionalClos}. 
Condition~\ref{ass:ConditionalClos} holds true if $\QP$ satisfies \emph{super-stability} and \emph{lower regularity} in the sense of Ruelle~\cite{Rue70, Osa98}, or the existence of upper semi-continuous bounds~$(\Phi_0,\Psi_0)$ such that
$$c\, \Phi_0\leq \Phi\leq c^{-1} \Phi_0\comma\quad c\,\Psi_0\leq \Psi \leq c^{-1}\Psi_0$$ for some constant~$c>0$, see~\cite[Eqn.~(A.3), p.~8]{Osa13}) and also~\cite{Osa96, Osa98}.  
\end{rem}

\subsection{Assumptions~\ref{ass:ConditionalErg}}
In this subsection, we verify Assumptions~\ref{ass:ConditionalErg}.
\begin{ass}\normalfont \label{asmp: ESL}
Let~$\mu$ be a quasi-Gibbs  measure on $\U$ satisfying~\ref{ass:ConditionalClos}, and suppose
\begin{enumerate}
\item there exists a closed $\mssm$-negligible set~$F\subset \R^n$ so that the free potential~$\Phi$ of~$\mu$ satisfies~$\Phi\in L^\infty_{loc}(\R^n\setminus F, \mssm)$;
\item there exists a closed $\mssm^{\otimes 2}$-negligible set~$F^{[2]}\subset \R^n \times \R^n$ so that the interaction potential~$\Psi$ of~$\mu$ satisfies~$\Psi\in L^\infty_{loc}\bigl(\R^n \times \R^n \setminus F^{[2]},\mssm^{\otimes 2} \bigr)$.
\end{enumerate}
\end{ass}

\begin{prop}[Sufficient conditions for~\ref{ass:ConditionalErg} {\cite[Prop.\ 7.13]{LzDSSuz21}}] \label{prop: ESL}
Under Assumption~\ref{asmp: ESL}, 
\ref{ass:ConditionalErg} holds.
\end{prop}
\begin{proof}
Noting that \ref{ass:ConditionalErg} follows from the conditional Sobolev-to-Lipschitz property proven in \cite[Prop.\ 7.13]{LzDSSuz21}, we conclude the statement. 
\end{proof}
\subsection{Markovian subspace with \ref{ass:QR} and ~\ref{p:Rad}}
The quasi-regularity~\ref{ass:QR} follows if  $\mathcal F^{\U, \QP}$ is chosen to be the closure of either: 
\begin{itemize}
\item $\Lip_b(\bar{\mssd}_\U, \tau_\mrmv)$ or $\Lip_b({\mssd}_\U, \tau_\mrmv)$ by~Cor.~\ref{cor:SDR}; 
\item smooth local functions $\mathscr D_\infty$ (see \cite[Thm.~1]{Osa96}).
\end{itemize}
If $\mathcal F^{\U, \QP}$ is chosen to be the closure of either 
$$\text{$\Lip_b(\bar{\mssd}_\U, \mu)\quad $ or $ \quad \Lip_b({\mssd}_\U, \mu)$} \comma$$
then Prop.~\ref{p:DF2} provides $\mathsf{(Rad_{\bar{\mssd}_\U, \QP})}$, $\mathsf{(Rad_{{\mssd}_\U, \QP})}$ respectively. With all these choices of cores, the Markovian property of $\mathcal F^{\U, \QP}$ has been proven in Prop.~\ref{p:DF2}. 
 \begin{cor}[Cor.~\ref{cor:SDR}, Prop.~\ref{p:DF2}]\label{c:QRR}
 Let $\QP$ be a quasi-Gibbs measure. 
 \begin{enumerate}[(i)]
 \item If either $\mathcal C=\Lip_b(\bar{\mssd}_\U, \tau_\mrmv)$, or $\mathcal C=\Lip_b({\mssd}_\U, \tau_\mrmv)$, then $\mathcal F^{\U, \QP}=\overline{\mathcal C}$ is Markovian and 
  $$\text{\ref{ass:QR} holds for $\mathcal F^{\U, \QP}$}\ ;$$
 \item If $\mathcal C=\Lip_b(\bar{\mssd}_\U, \mu)$ $($resp.~$\mathcal C=\Lip_b({\mssd}_\U, \mu)$$)$, then $\mathcal F^{\U, \QP}=\overline{\mathcal C}$ is Markovian and 
 $$\text{\ref{p:Rad} $($resp.~$\mathsf{(Rad_{{\mssd}_\U, \QP})}$$)$ holds for $\mathcal F^{\U, \QP}$}\fstop$$
 \end{enumerate}
 \end{cor}

\section{Examples} \label{sec: Exa}
Based on verifying the sufficient conditions provided in the previous section, we provide several examples to which our main results (Theorems~\ref{thm: Erg},~\ref{thm: Equiv}) can apply. 
In the following, we discuss four classes of examples: $\mathrm{sine}_2$, $\mathrm{Airy}_2$, $\mathrm{Bessel}_{\alpha,2}$ ($\alpha \ge 1$), and $\mathrm{Ginibre}$ point processes. 
They belong to the class of quasi-Gibbs measures as explained below, in particular, \ref{ass:SCE} holds true.  As all the examples discussed in the following are determinantal point processes, {the tail triviality}~\ref{ass:T} is a consequence of e.g., \cite[Theorem 2.1]{Lyo18} (see Example~\ref{exa: R} for more complete references). 
 
As noted in \ref{i:r:QuasiGibbs:3} in Remark \ref{r:QuasiGibbs},  the class of $(\Phi,\Psi)$-quasi-Gibbs measures is stable under perturbations of $\Phi$ in terms of adding locally bounded free potentials.  As seen in \cite[Thm.\ 2.2]{Osa13}, \cite[Thm.\ 5.6]{OsaTan14} and \cite[Thm.\ 2.3]{Osa13}, the free potentials $\Phi$ {\it representing} the classes of quasi-Gibbs measures in the cases of $\mathrm{sine}_2$, $\mathrm{Airy}_2$, and $\mathrm{Ginibre}$ are locally bounded, therefore $\Phi$ can be reduced to the representative
$$\Phi \equiv 0.$$
Thus, we only discuss the interaction potentials $\Psi$ for these cases below. 

\begin{ese}[$\mathrm{sine}_2$] \label{exa: S}
By \cite[Thm.\ 2.2]{Osa13}, the $\mathrm{sine}_2$ ensemble belongs to the class of $(0,\Psi)$-quasi-Gibbs measures with the interaction potential
$$\Psi(x, y):=-2\log|x-y|, \quad x, y \in \R.$$ 
Assumption~\ref{ass:ConditionalClos} follows from~\eqref{a:CCC}.
 Assumptions~\ref{ass:ConditionalErg} can be verified immediately by Proposition~\ref{prop: ESL} by noting that Assumption~\ref{asmp: ESL} is satisfied by taking $F^{[2]}=\{(x, y) \in \R^{\times 2}: x=y\}$ as $ \Psi \in L^\infty_{loc}(\R^{\times 2} \setminus F^{[2]}, \mssm^{\otimes 2})$. The number rigidity~\ref{ass:Rig} has been proved by \cite[Thm.\ 4.2]{Gho15} and \cite{ChhNaj18}. 
 A Markovian subspace $\mathcal F^{\U, \QP}$ having the quasi-regularity~\ref{ass:QR} and~\ref{p:Rad} has been constructed by Cor.~\ref{c:QRR}. We remark that the quasi-regularity~\ref{ass:QR} with respect to $\mathcal F^{\U, \QP}=\overline{\mathscr D}_\infty$  has been shown by combination of~~\cite[Cor.~4.1]{Osa13} and ~\cite[Thm.~1]{Osa96}.
\end{ese}

\begin{ese}[$\mathrm{Airy}_2$] \label{exa: A}
By \cite[Thm.\ 4.7]{OsaTan14}, the $\mathrm{Airy}_2$ ensemble belongs to the class of $(0,\Psi)$-quasi-Gibbs measures with the interaction potential
$$\Psi(x, y):=-2\log|x-y|, \quad x, y \in \R.$$
Thus the same arguments as in~Example~\ref{exa: S} apply to~\ref{ass:ConditionalClos}, ~\ref{ass:ConditionalErg} ~and~\ref{ass:QR}.
The number rigidity~\ref{ass:Rig} has been proved by~\cite{Buf16}. 
\end{ese}

\begin{ese}[$\mathrm{Bessel}_{\alpha,2}$, $\alpha \ge 1$]
By \cite[Thm.\ 2.4]{HonOsa15}, the class of measures $\mathrm{Bessel}_{\alpha,2}$ ($\alpha \ge 1$) belongs to the class of $(\Phi,\Psi)$-quasi-Gibbs measures with potentials (the sign of the potentials in \cite[Thm.\ 2.4]{HonOsa15} is opposite) 
$$\Phi(x)=-\alpha \log x, \quad \Psi(x, y):=-2\log|x-y|, \quad x, y \in \R.$$
Assumption~\ref{ass:ConditionalClos} follows from~\eqref{a:CCC}.
Assumptions~\ref{ass:ConditionalErg}  can be verified immediately by Proposition~\ref{prop: ESL} by the same argument in Example~\ref{exa: S} for $\Psi$. For $\Phi$, it suffices to take $F:=\{0\}$ in (i) in Assumption~\ref{asmp: ESL}, with which $\Phi$ belongs to $L^\infty_{loc}(\R \setminus F, \mssm)$.
The number rigidity~\ref{ass:Rig} has been proved by \cite{Buf16}.  
A Markovian subspace $\mathcal F^{\U, \QP}$ having the quasi-regularity~\ref{ass:QR} and~\ref{p:Rad} has been constructed by Cor.~\ref{c:QRR}. We remark that the quasi-regularity~\ref{ass:QR} with respect to $\mathcal F^{\U, \QP}=\overline{\mathscr D}_\infty$ has been shown by combination of~\cite[Thm.~2.4]{HonOsa15}, \cite[Lem.~2.1]{Osa13} and \cite[Thm.~1]{Osa96}. 
\end{ese}

\begin{ese}[$\mathrm{Ginibre}$]
By \cite[Thm.\ 2.3]{Osa13}, the class of measures $\mathrm{Ginibre}$ belongs to the class of $(\Phi,\Psi)$-quasi-Gibbs measures with the interaction potential 
$$\Psi(z_1, z_2):=-2\log|z_1-z_2|, \quad z_1, z_2 \in \R^{\times 2}.$$ 
Assumption~\ref{ass:ConditionalClos} follows from~\eqref{a:CCC}.
Assumptions~\ref{ass:ConditionalErg} can be verified immediately by Proposition~\ref{prop: ESL}. Note that Assumption~\ref{asmp: ESL} is satisfied since 
(ii) of Assumption~\ref{asmp: ESL} follows by taking $F^{[2]}=\{(x, y) \in (\R^{2})^{\times 2}: x=y\}$, with which $ \Psi \in L^\infty_{loc}( (\R^{2})^{\times 2} \setminus F^{[2]}, \mssm^{\otimes 2})$.  The number rigidity~\ref{ass:Rig} has been proved by \cite[Thm.~1.1]{GhoPer17}.  
A Markovian subspace $\mathcal F^{\U, \QP}$ having the quasi-regularity~\ref{ass:QR} and~\ref{p:Rad} has been constructed by Cor.~\ref{c:QRR}. We remark that the quasi-regularity~\ref{ass:QR} with respect to $\mathcal F^{\U, \QP}=\overline{\mathscr D}_\infty$  has been shown by combination of~\cite[Cor.~4.1]{Osa13} and \cite[Thm.~1]{Osa96}.

\end{ese}

\begin{appendix}
\section{}
\begin{lem}[{\cite[Lem.~A.1]{Suz22b}}] \label{l:WDG}
Let~$\QP$ be a Borel probability on~$\U$ satisfying that $\QP_r^\eta$ is absolutely continuous with respect to the Poisson measure~$\pi_{\mssm_r}$ for $r>0$ and $\QP$-a.e.~$\eta$. Let $\Sigma \subset B_r$ so that $\mssm_r(\Sigma^c)=0$. Let $\Omega(r):=\{\gamma \in \U: \gamma_\Sigma=\gamma_{B_r}\}$. Then, 
$$\QP\bigl(\Omega(r)\bigr)=1 \qquad r>0 \fstop$$
\end{lem}

Let $(\Omega,  \mathcal F, P)$ be a probability space. Recall that, for a sequence $(A_n)$ of sets in $\mathcal F$, we define the limit superior of sets as 
$$\limsup_{n \to \infty} A_n:=\bigcap_{n \ge 1} \bigcup_{j \ge n}A_j \fstop$$
Recall that, by a simple application of Fatou's lemma to the indicator function $\mathbf 1_{A_n^c}$, we see 
$$P(\limsup_{n \to \infty} A_n) \ge \limsup_{n \to \infty} P(A_n) \fstop$$
\begin{lem} \label{lem: IEF}
Let $(\Omega,  \mathcal F, P)$ be a probability space. 
Let $\{\Omega^r_m\}_{m, r \in \N} \subset \mathcal F$ satisfy that, 
for any $\epsilon>0$, there exists $m_\epsilon$ so that for every $m \ge m_\epsilon$ and $r \in \N$, 
$$P(\Omega^r_m) \ge 1-\epsilon \fstop$$ 
Then, there exists $n \mapsto m_n \in \N$ with $m_{n} \le m_{n'}$ for $n \le n'$ so that 
$$P\bigl(\limsup_{n \to \infty} \cap_{r=1}^n\Omega^r_{m_n} \bigr) = 1 \fstop$$
\end{lem}
\proof
Define $\Omega^{n, \epsilon}:=\cap_{r=1}^n \Omega^r_{m_\epsilon}$. Then, by a simple application of Inclusion-Exclusion formula and the hypothesis $P(\Omega^r_{m_\epsilon}) \ge 1-\epsilon$ for every $r \in \N$, it holds that 
\begin{align*}
P(\Omega^{n, \epsilon})  \ge 1-c(n) \epsilon \comma
\end{align*}
where $c(n)$ is a constant monotone increasing in $n$.  Let $C(n)$ be a monotone increasing sequence so that $c(n)/C(n) <1$ and $c(n)/C(n) \downarrow 0$ as $n \to \infty$. Take $\epsilon_n := \frac{1}{C(n)}$, and $\Omega^{n}:=\Omega^{n, \epsilon_n}$. 
By the upper semi-continuity of probability measures regarding the limit superior of sets, we obtain
$$P\bigl(\limsup_{n \to \infty}\Omega^{n} \bigr) \ge  \limsup_{n\to\infty}P\bigl(\Omega^{n} \bigr) = \limsup_{n\to\infty}P\bigl(\Omega^{n, \epsilon_n} \bigr) \ge \lim_{n \to \infty} 1-\frac{c(n)}{C(n)} = 1 \fstop$$
The proof is complete.
\qed
\end{appendix}
\bibliographystyle{alpha}
\bibliography{MasterBib.bib}

\end{document}